\documentclass[11pt]{amsart}
\usepackage{amsmath,amssymb,mathrsfs, xcolor, geometry}
\usepackage{hyperref}
\hypersetup{linkbordercolor = {white}, citebordercolor = {white},}
\usepackage[alphabetic]{amsrefs}

\numberwithin{equation}{section}

\theoremstyle{theorem}
\newtheorem{thm}[equation]{Theorem}
\newtheorem{prop}[equation]{Proposition} 
\newtheorem{cor}[equation]{Corollary}

\theoremstyle{definition}

\newtheorem{ques}[equation]{Question}
\newtheorem{lemma}[equation]{Lemma}
\newtheorem{remark}[equation]{Remark}

\newtheorem{claim}[equation]{Claim}

\newcommand{\ovl}[1]{\overline{#1}}

\newcommand{\pair}[1]{\langle #1\rangle}
\newcommand{\set}[1]{\left\{#1\right\}}
\newcommand{\ppair}[1]{\left\langle #1\right\rangle}
\newcommand{\pr}[1]{\left(#1\right)}
\newcommand{\abs}[1]{\left|#1\right|}
\newcommand{\bb}[1]{\mathbb{#1}} \newcommand{\td}[1]{\widetilde{#1}}

\newcommand{\tr}{\text{tr}}
\newcommand{\D}{\Delta}
\newcommand{\Hess}{\text{Hess}}
\newcommand{\n}{\nabla}

\newcommand{\sbst}{\subseteq}
\newcommand{\h}{\textbf H}

\newcommand{\bd}{\partial}

\title[Planarity and convexity for ancient MCF]
{Planarity and convexity for pinched ancient solutions of mean curvature flow
}

\author{Tang-Kai Lee}
\author{Keaton Naff}
\author{Jingze Zhu}

\address{MIT, Dept. of Math., 77 Massachusetts Avenue, Cambridge, MA 02139}
\address{Lehigh University, Dept. of Math., 17 Memorial Dr E, Bethlehem, PA 18015}
\address{MIT, Dept. of Math., 77 Massachusetts Avenue, Cambridge, MA 02139}

\email{tangkai@mit.edu, ken424@lehigh.edu,  zhujz@mit.edu}

\date{\today}

\begin{document}
\begin{abstract}
We prove a parabolically scale-invariant variation of the planarity estimate in \cite{Na22} for higher codimension mean curvature flow, borrowing ideas from work of Brendle--Huisken--Sinestrari \cite{BHS}. Additionally, we prove convexity for pinched complete ancient solutions of the mean curvature flow in codimension one. Then we put these estimates together to characterize certain pinched complete ancient solutions and shrinkers in higher codimension. We include some discussion of future research directions in this area of mean curvature flow. 
\end{abstract}
\maketitle

\section{\textbf{Introduction}}

In this paper, we prove rigidity theorems for a class of ancient solutions of mean curvature flow (MCF) 
satisfying a curvature pinching condition. 
They can be viewed as improved versions of the planarity estimate of the second author~\cite{Na22} and the convexity estimate of Huisken--Sinestrari~\cite{HS99}.

Consider integers $n$ and $N$ so that $N\ge n+1$. 
In what follows, a (higher codimension) solution of MCF is a family of immersions, $F : M^n \times I \to \mathbb{R}^N$, evolving in the direction of its mean curvature vector, $\partial_t F = H$, which we often denote briefly by writing $M^n_t = F(M^n, t) \subset \mathbb{R}^N$, $t \in I$. 
We say a MCF $M^n_t$ is \textit{uniformly} $c$-pinched if its second fundamental form and mean curvature satisfy the estimate 
$$|A|^2 \le (c-\epsilon_0) |H|^2$$ 
for some $\epsilon_0>0.$
A geometrically interesting choice for $c$ is $c = \frac{1}{n-k}$ for $k \in \{1, \dots, n-1\}$, which can be viewed as a higher dimensional analogue of the condition of $k$-convexity in codimension one. 
We will review geometric significance of this pinching condition, and refer the reader to the introduction of \cite{Na22} for a more detailed discussion of this pinching condition and MCF.

\subsection{Planarity estimate for ancient solutions}
\label{sec:1.1}

To state our first main result, we recall from \cite{Na22} that whenever the mean curvature is non-vanishing (i.e., $|H| > 0$) we can define the principal unit normal vector to $M^n_t$ given by $\nu_1 := |H|^{-1} H$. With respect to this principal normal, the (vector-valued) second fundamental decomposes as a sum $A_{ij} = \hat{A}_{ij} + h_{ij} \nu_1$ of a principal component given by the (scalar-valued) $(0, 2)$-tensor  $h_{ij} = \langle A_{ij},  \nu_1\rangle$ and its (vector-valued) orthogonal complement, which we denote here by $\hat{A}_{ij}$. The vanishing of the tensor $\hat{A}$ is related to planarity of the submanifolds $M^n_t$ \cite[Proposition 2.5]{Na22}. 
Our first main result is the following theorem, where we let  
$$c_n := \min \set{\frac{4}{3n}, \frac{3(n+1)}{2n(n+2)}}
= \begin{cases}
    \frac{3(n+1)}{2n(n+2)}&\text{if }n\le 7\\
    \frac 4{3n}&\text{if }n\ge 8
\end{cases}
$$
in the rest of the paper.

\begin{thm}\label{thm:BHS-type-est}
	Suppose $n \geq 2$ and $N > n + 1$. 
	Let $M^n_t = F(M^n, t) \subset \mathbb{R}^N$, $t\in[0,T]$, be a smooth family of $n$-dimensional, complete, connected, immersed submanifolds evolving by MCF.
	Suppose each time slice $M_t$ has bounded curvature.
	Finally, suppose there exists a constant $\epsilon_0 \in (0, c_n)$ such that $|A|^2 \leq (c_n-\epsilon_0)|H|^2$ holds on $M^n \times [0, T]$. Then either $M^n_t$ is a static, flat copy of $\mathbb{R}^n$ in $\mathbb{R}^N$, or else $|H| > 0$ on $M \times (0, T]$ and there exist constants $C=C(n,\epsilon_0)$ and $\sigma = \sigma(n,\epsilon_0)$ such that
\begin{equation*}
	|\hat A|^2
	\le C \frac{|H|^{2-\sigma}}{t^{\sigma/2}}
\end{equation*}
	holds pointwise on $M \times (0, T]$. 
\end{thm}

To contextualize the result above, we offer the following remarks:
\begin{itemize}
\item The main result of \cite{Na22} asserts that whenever $M_t^n$ is a compact flow such that its initial data $M_0^n$ satisfies the strict pinching $|A|^2 < c_n |H|^2$, then there exist constants $C = C(M_0), \sigma = \sigma(M_0)$ such that $|\hat {A}|^2 < C |H|^{2-\sigma}$. Compared to this original planarity estimate, the estimate above is \textit{parabolically scale-invariant} and the \textit{constants depend only upon the pinching}, not on the initial submanifold.  
\item Although the planarity estimate proved in \cite{Na22} is only asserted for $n \geq 5$, in fact the proof works for $n \geq 2$ (see Appendix \ref{app:planarity}), and so we assert Theorem \ref{thm:BHS-type-est} for $n \geq 2$. However, we remark that Andrews and Baker have proved a stronger estimate when $n \in \{2, 3, 4\}$. Indeed, for such $n$, one has that $c_n \leq \min\{\frac{4}{3n}, \frac{1}{n-1}\}$ which by the Theorem 4 of \cite{AB} implies the stronger estimate $|\mathring{A}|^2 \leq C |H|^{2-\sigma}$. This estimate from \cite{AB} is not parabolically scale-invariant like planarity estimate above, but a scale-invariant version of Andrews and Baker's estimate can be found in \cite{LN21}, building upon the codimension one work \cite{HS15}. 
\item The theorem above essentially implies the main result of \cite{Na22}. Indeed, one can use the estimate above to obtain the bound $|\hat{A}|^2 \leq C(n, \epsilon_0, \delta) |H|^{2-\sigma}$ for $t  \in [\delta, T]$, $\delta > 0$ and then bound $|\hat{A}|^2/ |H|^{2-\sigma}$ on the compact set $M \times [0, \delta]$ to cover the remaining times. 
\item The constant $c_n$ is a technical one that arises by the maximum principle bases proofs of 
\begin{itemize}
    \item the preservation of $c$-pinching in MCF, done by Andrews and Baker \cite{AB}, which requires $c \leq \frac{4}{3n}$; and
    \item the original planarity estimate, which requires $c \leq c_n$. 
\end{itemize}
Further discussion of precisely how the constant $c_n$ arises can be found in Appendix \ref{app:planarity}.
\end{itemize}

An important consequence of the main result above is the following rigidity theorem for ancient pinched solutions of MCF. 
To state it in the most general case, we mention that the pointwise estimate that leads to Theorem~\ref{thm:BHS-type-est} can be used to obtain an integral version of the improved planarity estimate, which we record in Theorem~\ref{thm:BHS-type-est-int-version}.
Because of the Gaussian weight in this version of the estimate, we can consider flows with unbounded curvature of polynomial growth, leading to the following rigidity theorem.
We remark that a planarity result for ancient MCFs was proven by Colding--Minicozzi~\cite{CM20} under a different assumption; instead of pinching assumption, they assume that the blow-down limit is a multiplicity one generalized cylinder.

\begin{cor}\label{cor:complete-bdd-curv-pinched}
	Let $M^n_t = F(M, t) \subset \mathbb{R}^N$, $t \in (-\infty, 0)$, be an ancient smooth family of $n$-dimensional, complete, connected, immersed submanifolds evolving by MCF with finite entropy.  
	Suppose each $M_t$ has curvature of polynomial growth. 
    If $M_t$ is uniformly $c_n$-pinched,
    then the flow is codimension one for all time.
\end{cor}

From a more geometric viewpoint, for $k \in \{0, 1, \dots, n-1\}$, we observe that a family of (codimension one) shrinking cylinders $\mathbb{S}^{n-k}\pr{\sqrt{-2(n-k)t}} \times \mathbb{R}^{k}$ 
satisfies $|A|^2 = \frac{1}{n-k} |H|^2$. 
Moreover, the codimension one, noncollapsed, ancient solutions of mean curvature flow that have such a cylinder as a blow-down satisfy $|A|^2 < \frac{1}{n+1-k} |H|^2$. 
Since $\frac{1}{n-k} \leq c_n$ whenever $n \geq 4k$ and $k \geq 2$, or  $n \geq 5$ if $k = 1$, or $n \geq 1$ if $k =0$, the class of uniformly $c_n$-pinched ancient solutions contains many of the well-studied codimension one noncollapsed ancient solutions for $n$ sufficiently large.

Note that based on \cite{Na22}, the blow-up limits (i.e., ancient solutions that arise as parabolic rescalings of the flow) at the first singular time of a closed strictly $c_n$-pinched MCF must be codimension one. Corollary~\ref{cor:complete-bdd-curv-pinched} deals with general smooth ancient solutions, which, a priori, forms a possibly distinct class of solutions from the class of blow-up limits.\footnote{For instance, in codimension one, it is not yet known whether ancient ovals can arise as blow-up limits after the first singular time of a MCF. The behavior would be fairly pathological, cf. \cite{CH21}} It is not yet known whether the blow-up limits of a closed $c_n$-pinched flow must have bounded curvature, nor whether every ancient solution with bounded curvature arises as a blow-up limit. A natural question is whether the results above can hold without the assumption of bounded curvature. This seems related to recent work about the existence of local Ricci flow established in \cite{LT22}. In Section \ref{sec:Question}, we list several open questions in this area of mean curvature.

Let us now briefly describe the proof of the improved planarity estimate.
In his ground-breaking work on Ricci flow \cite{Ham82}, Hamilton derived the evolution equation for $u = |\mathring{\mathrm{Ric}}|^2/ R^{2-\sigma}$ to establishing a pinching estimate for compact solutions of Ricci flow in dimension three that have positive Ricci curvature. Subsequently, by exploiting the structure of the equation for $(\partial_t - \Delta) u$, Brendle--Huisken--Sinestrari proved a parabolically scale invariant version of Hamilton's estimate and used it to prove a rigidity theorem ancient solutions of Ricci flow \cite{BHS}. Similar work was later done for MCF by Huisken--Sinestrari \cite{HS15}. We use similar ideas and exploit the evolution for $u = |\hat{A}|^2/((c_n-\epsilon_0) |H|^2 - |A|^2)$ studied in \cite{Na22} to achieve our results.

\subsection{Convexity estimate for ancient solutions}

Our next result investigates convexity of mean convex ancient solutions.
This is related to Corollary~\ref{cor:complete-bdd-curv-pinched}, which implies that every uniformly $c_n$-pinched ancient MCF with finite entropy and polynomial growth curvature is codimension one, and hence a mean convex MCF since $c_n\le 1.$
The convexity of ancient solutions of MCF is also an important question on its own, cf. \cite[Open Problems 14.35 and 14.36]{ACGL}.

Our second main result proves the convexity of a pinched ancient MCF with bounded entropy and a curvature growth assumption.
Different from the planarity estimate, we can instead prove the convexity estimate for any $L$-pinched flows with \textit{arbitrary} $L<\infty.$

\begin{thm}
\label{thm:improved-conv}
	Let $M_t\sbst\bb R^{n+1}$ be an ancient mean convex codimension one MCF with finite entropy.
    Suppose each $M_t$ has curvature of polynomial growth.
	If $M_t$ is uniformly $L$-pinched for some $L>0$, then $M_t$ is convex.	
\end{thm}

The pinching assumption is necessary since there are examples of non-convex mean convex ancient solutions coming out from some minimal hypersurfaces, constructed by Mramor--Payne \cite{MP21}.
It is also known that such an ancient solution can be produced from a minimal hypersurface with finite total curvature \cite{CHL}.
It will be interesting to see whether these flows are the only non-convex mean convex non-trivial ancient solutions besides the trivial static flows given by minimal hypersurfaces.

Theorem~\ref{thm:improved-conv} can be interpreted as an improvement estimate which shows that any $L$-pinched ancient mean curvature flow with finite entropy has to be $1$-pinched (possibly not uniformly).
This can be compared with the improved non-collapsing estimate for mean convex ancient solutions, see~\cite{HK15}, where they use a compactness argument to show that any $\alpha$-noncollapsed ancient mean curvature flow has to be $1$-noncollapsed.
Applying the methods to prove Theorem~\ref{thm:improved-conv} in Section~\ref{sec:conv} and using Brendle's estimates in~\cite{B15} may provide another proof of this improvement result based on a PDE method.

Theorem~\ref{thm:improved-conv} is a convexity-improvement estimate for ancient solutions. The earliest work on convexity estimates was done for blow-up limits.
In \cite{HS99a, HS99}, Huisken--Sinestrari established the first convexity estimate, which implies that all blow-up limits at the first singular time of a mean convex MCF must be convex.
White~\cite{W03} used a different method to prove the convexity of blow-up limits of a mean convex MCF, and he can extend it to subsequent singular times based on his geometric measure techniques.

As for ancient solutions, the earliest convexity estimates exploited non-collapsing. See for instance ~\cite{W03, HK}, and also ~\cite{L17,BC19, ADS20, BC21}.
The results in \cite{L17} are most related to ours.
In \cite{L17}, inspired by the ideas in \cite{HS15} (cf. \cite{HH16}), Langford showed that a pinched convex ancient MCF with \textit{bounded polynomially rescaled volume} \cite[Equation~(1.15)]{L17} must be convex. 
Bounded rescaled volume is a mild condition, and it is interesting to study the connection between it and the bounded entropy assumption, which is another natural condition in the study of MCF singularities.

The main ingredient that allows us to deal with the noncompact case in Theorem~\ref{thm:improved-conv} is to consider integral estimates with the Gaussian weight
$$\Phi(x,t)
:= \frac 1{(-4\pi t)^{n/2}}
e^{\frac{|x|^2}{4t}}.$$
In the original convexity estimate in \cite{HS99}, integral estimates already play an essential role, since evolution equations in MCF usually have worse reaction terms, preventing us from applying the maximum principle directly, see Remark~\ref{rmk:MCF-evol}.
However, standard integral estimates cannot deal with noncompact flows, even with a bounded curvature assumption.
Thus, we introduce the Gaussian weight and study the evolution of the total weighted integral of $\lambda_1,$ the smallest principal curvature of the hypersurface.
The weight complicates the growth estimate of $\lambda_1,$ but it can successfully reduce the problem to the study of an ordinary differential equation.
Careful analysis implies that such an ODE only has trivial non-negative ancient solutions, which leads to our convexity estimate $\lambda_1\ge -\varepsilon H$ for all $\varepsilon>0.$

\subsection{Consequences of the main estimates} 
\label{sec:1.3}

Corollary~\ref{cor:complete-bdd-curv-pinched} and Theorem~\ref{thm:improved-conv} are expected to have interesting applications to the understanding of flow behavior near singularities.
First, we can combine Theorem~\ref{thm:improved-conv} with the main theorem in \cite{BLL} to obtain:

\begin{cor}
\label{cor:convex-then-noncollapsing}
	Let $M_t\sbst \bb R^{n+1}$ be an ancient mean convex codimension one MCF with finite entropy.
	Suppose each $M_t$ has curvature of polynomial growth.
	If $M_t$ is uniformly $1$-pinched, then $M_t$ is a non-collapsed convex ancient solution.
\end{cor}

Indeed, if $M_t$ is uniformly $1$-pinched, then by Theorem \ref{thm:improved-conv}, $M_t$ must be convex. Since uniform $1$-pinching implies uniform $(n-1)$-convexity in codimension one, the result now follows from \cite[Theorem 1.1]{BLL} (see also \cite{W11}). 
Other works related to establishing non-collapsing on ancient solutions may be found in \cite{Na-Sing} (which considers the relationship between pointwise gradient estimates and non-collapsing in the two-convex setting) and a local non-collapsing result \cite{BN24}, which provides another route towards non-collapsing when the blow-down of an ancient solution is known to be a generalized cylinder.

By combining Corollary~\ref{cor:convex-then-noncollapsing} with Corollary~\ref{cor:complete-bdd-curv-pinched} and previous classificaiton results in \cite{BC19, BC21, ADS20},
we get the following classification results for $c_n$-pinched flows.

\begin{cor}
\label{cor:BC-ADS-HS-LN}
	Let $M_t\sbst  \bb R^N$ be an ancient smooth family of $n$-dimensional, complete, immersed submanifolds evolving by MCF.
	If $M_t$ has bounded entropy, is uniformly $c_n$-pinched, and each $M_t$ has curvature of polynomial growth, then $M_t$ is a non-collapsing convex ancient solution.
	Moreover: \\
	(1) If $n\ge 5$ and the flow is uniformly $\frac 1{n-2}$-pinched, then the flow is one of the shrinking sphere, the shrinking cylinder, the translating bowl soliton,  the ancient oval, and the static plane.\\
    (2) If $n\ge 5$ and the flow is uniformly $\frac 1{n-1}$-pinched, then the flow is either the shrinking sphere or the static plane.\\
    (3) If $n\in\set{2,3,4},$ then the flow is either the shrinking sphere or the static plane.
\end{cor}

We obtain the first part of the corollary based on the integral planarity estimate and the classification results in \cite{BC21, ADS20}.
For the second part, when the flow is compact, it was discussed in \cite{LN21}, where Lynch--Nguyen studied compact uniformly $\min\set{\frac1{n-1},\frac 4{3n}}$-pinched flows.
Here, by observing that their compactness criterion can be generalized by a recent result of M.-C. Lee and Topping~\cite{LT22-PIC}, we are able to remove the compactness assumption.
In fact, in the low dimensional case when $n\in\set{3,4,5,6,7},$ we can instead use the better constant $\min\set{\frac1{n-1},\frac 4{3n}}$.
See the remark after Corollary~\ref{cor:cyl-shrinker}.

A version of Corollary \ref{cor:BC-ADS-HS-LN} is expected to hold for $\frac{1}{n-3}$-pinched flows as well. There has recently been a lot of progress in the classification program of non-collapsed singularities in $\bb R^4$ (see \cite{Z22, CHH23, CHH24, DH24, CDDHS22}), which was finally concluded by Choi--Haslhofer~\cite{CH24}.
If one can obtain corresponding classification results for uniformly three-convex ancient solutions in~$\bb R^{n+1}$ (see 
\cite{Z21, DZ22, CDZ25} for partial progresses), then our rigidity result will also lead to the corresponding classification results for uniformly $\min\set{\frac 1{n-3},c_n}$-pinched flows.

Next, we can use Corollary~\ref{cor:complete-bdd-curv-pinched} and Theorem~\ref{thm:improved-conv} to generalize previous rigidity theorems for spheres and cylinders. 
As an example, we can generalize \cite[Theorem~4.2]{LN21} to the non-compact case, or generalize a rigidity result \cite[Theorem~1.1]{L22} to the high codimension case.
We can even relax the pinching condition and allow the possibly largest pinching constant $c_n.$
We note that the constant $c_n$ in \cite{LN21}, which is $\min \set{\frac{4}{3n}, \frac{1}{n-1}}$, is smaller than the one considered in this note. 

\begin{cor}
\label{cor:type-I-pinched-generalized-cylinder}
	Let $M_t\sbst  \bb R^N$ be an ancient smooth family of $n$-dimensional, complete, immersed submanifolds evolving by MCF. 
	Suppose $M_t$ is uniformly $c_n$-pinched.
	If $M_t$ is of type-I, in the sense that $|A|\le C/\sqrt{-t}$ on $M^n \times (-\infty, 0)$ for some $C<\infty,$ then the flow is $\mathbb{S}^{n-k}\pr{\sqrt{-2(n-k)t}} \times \mathbb{R}^{k}$ for some $k=0,\cdots,n.$
\end{cor}

Indeed, by Corollary ~\ref{cor:complete-bdd-curv-pinched} and Theorem~\ref{thm:improved-conv}, for $M_t$ as above we conclude $M_t$ is codimension one and convex. Hence, by \cite[Theorem~1.1]{L22}, the corollary follows. 

The pinching constant $c_n$ is due to the planarity estimate.
It can be weakened to an arbitrary pinching condition based on Theorem~\ref{thm:improved-conv} in the hypersurface case.
As mentioned in Corollary~\ref{cor:BC-ADS-HS-LN}, in the low dimensional case, the constant can be weakened and the resulting flow can only be the compact shrinking sphere.
Other generalizations can be done based on \cite[Corollaries~1.5 and 1.6]{L17}.

Finally, the main theorems also lead to a new classification result of high codimensional shrinkers.
Based on Theorem~\ref{thm:BHS-type-est-int-version}, a uniformly $c_n$-pinched shrinker with bounded entropy must be codimension one, and based on Theorem~\ref{thm:improved-conv}, it must be convex.
We can then argue that such a shrinker must be a generalized cylinder. 

\begin{cor}\label{cor:cyl-shrinker}
    Let $n\ge 2.$
    A uniformly $c_n$-pinched shrinker with bounded entropy must be a generalized cylinder $\mathbb{S}^{n-k}\pr{\sqrt{2(n-k)}} \times \mathbb{R}^{k}$  for some $k=0,\cdots,n.$
\end{cor}

We give some remarks about the low dimensional case.
When $n=1,$ all shrinkers are classified, and they are the straight line and the Abresch--Langer curves \cite{AL}.
When $n\in\set{3,4},$ since 
$$c_n= \frac{3(n+1)}{2n(n+2)}
<\frac 4{3n}
<\frac 1{n-1}
<\frac 1{n-2},$$ 
the only non-planar shrinker allowed in Corollary~\ref{cor:cyl-shrinker} 
is the round sphere, based on the compactness criterion, Proposition~\ref{prop:n-1-pinched-compact},
which then allows us to put the best constant $\frac 4{3n}$ (instead of $c_n$) given by the work~\cite{LN21}.
In general, as mentioned above, a $k$-cylinder is allowed if $k\le n/4.$

In~\cite[Main Theorem~8]{B11}, Baker proved Corollary~\ref{cor:cyl-shrinker} under a bounded curvature assumption.
Corollary~\ref{cor:cyl-shrinker} then shares the same spirit as in \cite{CM12}, where Colding--Minicozzi generalized Huisken's classification of mean convex shrinkers by removing a bounded curvature assumption, and they also used integral curvature estimates to achieve this.
Here, we use the integral estimate to deal with shrinkers with possibly unbounded curvature, but thanks to the shrinker equation and the pinching condition, a shrinker automatically has curvature of polynomial growth.
This allows us to deal with its weighted integral without any further restriction.

We organize the paper as follows.
In Section~\ref{sec:Planarity}, we prove the improved pointwise and integral planarity estimates, Theorem~\ref{thm:BHS-type-est}
and Corollary~\ref{cor:complete-bdd-curv-pinched}.
In Section~\ref{sec:conv}, we prove the convexity estimate, Theorem~\ref{thm:improved-conv}.
In Section~\ref{sec:corollary}, we provide the proofs of Corollaries~\ref{cor:BC-ADS-HS-LN} and~\ref{cor:cyl-shrinker}.
In Section~\ref{sec:Question}, we list some open questions related to pinched MCF in higher codimension.

\subsection*{\bf Acknowledgment}
The authors are grateful to Stephen Lynch, Jonathan J. Zhu, and Bill Minicozzi for several helpful discussions. Part of this work was completed while KN was supported by the National Science Foundation under grant DMS-2103265. 

\section{\bf Planarity estimate}
\label{sec:Planarity}

\textit{On notation in Section \ref{sec:Planarity}:} In this section, all MCFs are, a priori, higher codimension and by convention we let $H = H^\alpha \nu_\alpha$ denote the vector-valued mean curvature and $A = A_{ij}^\alpha dx^i \otimes dx^j \otimes \nu_\alpha$ denote the vector-valued second fundamental form. When $|H| > 0$, we let $\nu_1 := |H|^{-1} H$ and express $A = h \nu_1 + \hat{A}$ as the orthogonal sum of the principal (codimension one) component $h := \langle A, \nu_1 \rangle$ and its complement $\hat{A}$. If we let $D$ denote the usual connection, then we let $\nabla := D^\top$ denote the induced on $TM$ and $T^\ast M$ and tensor bundles $T^\ast M^{\otimes p} \otimes TM^{\otimes q}$ made from these. We let $\nabla^\perp := D^\perp$ denote the connection on the normal bundle $T^\perp M$, and the induced connection on tensors taking values in the normal bundle $ T^\ast M^{\otimes p} \otimes TM^{\otimes q}\otimes (T^\perp M)^{\otimes k}$. We also let $\partial_t^\perp T:= (\partial_t T)^\perp$ for any evolving normal-vector-valued tensor $T$. These conventions are chosen to match those in \cite{Na22}. 

\subsection{Pointwise estimate}
\label{sec:pointwise-planarity}
Our proof of Theorem \ref{thm:BHS-type-est} is based on the estimates in \cite{Na22} and the techniques in \cite{BHS} and \cite{L24}. 
Before we start the proof, we give a remark about the evolution equations in common geometric flows.

\begin{remark}\label{rmk:MCF-evol}
In many cases in the MCF setting, maximum principle type arguments (that work, say, in Ricci flow) must be replaced by more subtle Stampacchia iterations arguments because the reaction terms in curvature PDEs in MCF are a bit less positive and hence exhibit weaker pinching  (cf. \cite{Ham82} and \cite{H84} for the classical instance of this phenomenon). 
In the present setting, however, the techniques of \cite{BHS} and \cite{L24} (which are maximum principle arguments in the Ricci flow setting) carry over without the introduction of a more complicated iteration scheme. 
This is because the PDE for $|\hat{A}|^2$, which is used here as the measure of planarity, is better behaved than other curvature quantities in MCF (such as the norm of the traceless second fundamental form $|\mathring{h}|^2$ in \cite{H84} or the first eigenvalue of the second fundamental form $\lambda_1$ in \cite{HS99}). 
When $|\hat{A}|^2$ is small, the reaction terms in the PDE for $|\hat{A}|^2$ are quadratically small (of order $\sim |\hat{A}|^4$), which distinguishes this equation from others in MCF. 
\end{remark}

Now suppose $M^n_t = F(M^n, t) \subset \mathbb{R}^N$ is a smooth, complete, connected, immersed MCF for $t \in [0, T]$ such that
\begin{itemize}
\item each time slice $M^n_t$ has bounded curvature, and 
\item  $|A|^2 \leq (c_n - \epsilon_0)|H|^2$ holds pointwise on $M \times [0, T]$. 
\end{itemize}
Define a constant $c_0:=c_n-\frac{1}{2}\epsilon_0$ if $n=5,6,7$ and $c_0:=c_n$ if $n\ge 8.$
Let  
\begin{equation}
f:=c_0|H|^2 - |A|^2.
\end{equation}
Then for any $n$, by assumption, 
\begin{equation}
f \geq \left(c_n-\frac{1}{2}\epsilon_0\right) |H|^2 - |A|^2 \ge \frac{1}{2}\epsilon_0 |H|^2 \ge 0.
\end{equation}
Let us first show that $|H| > 0$ on $M \times (0, T]$ unless $M^n_t$ is a flat, static copy of $\mathbb{R}^n$.

We first state a minor extension of \cite[Lemma 4.5]{Na22}.

\begin{lemma}\label{lem:evol-of-f}
The function $f$ evolves according to the equation 
\begin{equation}\label{eq:evol-of-f}
(\partial_t - \Delta) f = 2\left(c_0 |\langle A, H \rangle|^2 - |\langle A, A \rangle|^2 - |R^\perp|^2\right) + 2 (|\nabla^\perp A|^2 - c_0 |\nabla^\perp H|^2). 
\end{equation}
At any point $(p, t) \in M \times (0, T]$ where $|H|(p, t) > 0$, 
\begin{equation*}
(\partial_t -\Delta)f \geq 2\left( |h|^2 + \frac{1}{nc_0 -1} |\hat{A}|^2\right) f + 2 \left(1 - \frac{c_0(n+2)}{3}\right)|\nabla^\perp A|^2 \geq 0.
\end{equation*}
At any point $(p, t) \in M \times (0, T]$ where $|H|(p, t) = 0$, we have $|A|(p, t)=0$, hence
\begin{equation*}
(\partial_t-\Delta)f \geq 2 \left(1 - \frac{c_0(n+2)}{3}\right)|\nabla^\perp A|^2 \geq 0.
\end{equation*}
In particular, since by assumption $f \geq 0$ on $M \times [0, T]$, we have $(\partial_t -\Delta)f \geq 0$ on $M \times (0, T]$. 
\end{lemma}
\begin{proof}
Equation \eqref{eq:evol-of-f} is a direct consequence of combining the evolution equations of $|A|^2$ and $|H|^2$, which may be found in Proposition 2.3 of \cite{Na22}. In Lemma 4.5 of \cite{Na22}, it is shown\footnote{The lemma in \cite{Na22} is mistakenly missing a factor of $2$ on the gradient term, but this only makes the inequality stronger. Additionally, in the setting of \cite{Na22}, although it is not written, it is implicitly assumed in Lemma 4.5 that $|H| > 0$. In the present setting, we consider the possibility that $|H|= 0$, which is why we restate the lemma.} that if $|H| > 0$ and $f \geq 0$, then
\[
(\partial_t - \Delta) f \geq 2 |h|^2 f + \frac{2}{nc_0 -1} |\hat{A}|^2 f +2\left(1 - \frac{c_0(n+2)}{3}\right)|\nabla^\perp A|^2 \geq 0.
\]
On the other hand, when $|H| = 0$, by assumption we have $f = -|A|^2 \geq 0$, which is only possible if $A \equiv 0$. It follows then (see (4.18) in \cite{Na22}) that 
\[
(\partial_t - \Delta) f = 2(|\nabla^\perp A|^2 - c_0 |\nabla^\perp H|^2) \geq 2\left(1 - \frac{c_0(n+2)}{3}\right)|\nabla^\perp A|^2  \geq 0.
\]
In either case, we find that under the assumption $f \geq 0$ on $M \times[0, T]$, we have $(\partial_t - \Delta) f \geq 0$ on $M \times (0, T]$. 
\end{proof}

As a corollary to the above lemma, since $(\partial_t - \Delta) f \geq 0$ and $f \geq 0$, if $f(p_0, t_0) = 0$ for some $(p_0, t_0) \in M \times (0, T]$, then by the strong maximum principle $f \equiv 0$ on $M \times [0, t_0]$. If this holds, then $A \equiv 0 $ on $M \times [0, t_0]$ and so $M \times [0, T]$ is a static copy of $\mathbb{R}^n$ in $\mathbb{R}^N$. Otherwise, we have $f > 0$ on $M \times (0, T]$ and hence $|H| > 0$ on $M \times (0, T]$, which we assume from now on. This justifies existence of a principal normal $\nu_1 := |H|^{-1} H$ on $M \times (0, T]$ and hence the decomposition $A = h \nu_1 + \hat{A}$ in ensuing computations. 

\begin{proof}
[Proof of Theorem \ref{thm:BHS-type-est}]
    As mentioned, it remains to consider the case when $f>0$ on $M\times (0, T].$ 
    If we take $\delta > 0$ small enough, depending upon $n$ and $\epsilon_0$, then Lemmas 4.4 and 4.9 in \cite{Na22} imply
    \begin{align}\label{Na22-435}
    \frac{1}{f}(\partial_t -\Delta)|\hat{A}|^2  - \frac{|\hat{A}|^2}{f^2}(\partial_t - \Delta)f \leq - \delta \frac{|\hat{A}|^2}{f^2} (\partial_t - \Delta)f.
    \end{align}
    Specifically,\footnote{Note that our definition of $\epsilon_0$ in this note differs from the one in \cite{Na22} by a factor of $2$.} this holds if we take $\delta := \min \{\frac{1}{2}, \frac{n(n+2)}{3(n-1)}\epsilon_0 \}$ in dimensions  $n \in \{5, 6, 7\}$ and $\delta := \frac{1}{5n-8}$ for $n \geq 8$.
    Direct calculations give
    \begin{align*}
    (\bd_t-\D) f^{1-\sigma}
    = (1-\sigma) \frac 1{f^\sigma} (\bd_t-\D)f
    + \sigma(1-\sigma) \frac{1}{f^{1+\sigma}} |\n f|^2,
    \end{align*}
    which implies
    \begin{align*}
    (\bd_t-\D)\frac{|\hat A|^2}{f^{1-\sigma}}
    & = \frac{1}{f^{1-\sigma}} (\bd_t-\D) |\hat A|^2
    - \frac{|\hat A|^2}{f^{2-2\sigma}} (\bd_t-\D)f^{1-\sigma}
    + 2\ppair{\n \frac{|\hat A|^2}{f^{1-\sigma}}, \n \log f^{1-\sigma}}\\
    & = \frac{1}{f^{1-\sigma}} (\bd_t-\D) |\hat A|^2
    + 2\ppair{\n \frac{|\hat A|^2}{f^{1-\sigma}}, \n \log f^{1-\sigma}}\\
    &\quad - \frac{|\hat A|^2}{f^{2-2\sigma}}
    \pr{
    (1-\sigma) \frac 1{f^\sigma} (\bd_t-\D)f
    + \sigma(1-\sigma) \frac{1}{f^{1+\sigma}} |\n f|^2
    }\\
    & = f^\sigma\pr{
    \frac 1f (\bd_t-\D) |\hat A|^2
    - \frac{|\hat A|^2}{f^2}(\bd_t-\D)f
    }
    + \sigma \frac{|\hat A|^2}{f^{2-\sigma}} (\bd_t-\D)f\\
    &\quad 
    + 2\ppair{\n \frac{|\hat A|^2}{f^{1-\sigma}}, \n \log f^{1-\sigma}}
    - \sigma(1-\sigma) \frac{|\hat A|^2}{f^{3-\sigma}} |\n f|^2.
    \end{align*}
    Thus, using \eqref{Na22-435} above, if we let 
    $$u:= \frac{|\hat A|^2}{f^{1-\sigma}},$$ 
    then we have
    \begin{align*}
    (\bd_t-\D) u
    \le (\sigma-\delta) \frac uf (\bd_t-\D) f
    + 2\pair{\n u, \n \log f^{1-\sigma}}
    - \sigma(1-\sigma)u\frac{|\n f|^2}{f^2}.
    \end{align*}
    In \cite{Na22}, we took $\sigma$ equal to $\delta$ and threw away the last gradient term to obtain the original planarity estimate. If we instead take $\sigma = \sigma(n, \epsilon_0) :=\delta/2$ (following the idea in \cite{BHS}), we get
    \begin{align}\label{u-evol-1}
    (\bd_t-\D) u
    \le -\sigma \frac uf (\bd_t-\D) f
    + 2\pair{\n u, \n \log f^{1-\sigma}}
    - \sigma(1-\sigma)u\frac{|\n f|^2}{f^2}.
    \end{align}
    Now observe that 
    \[
    |\hat{A}|^2  \leq |A|^2 \leq c_0 |H|^2 \leq 2c_0 \epsilon_0^{-1} f, 
    \]
    which implies $u \leq 2 c_0 \epsilon_0^{-1}f^{\sigma}$. 
    Since $f>0,$ Lemma \ref{lem:evol-of-f} above therefore implies 
    \begin{align*}
    \frac 1f (\bd_t-\D) f
    \ge \frac 2{nc_0-1}|\hat A|^2
    = \frac 2{nc_0-1} \frac{|\hat{A}|^2}{f^{1-\sigma}} (f^\sigma)^{\frac{1}{\sigma}-1} \ge \frac 2{nc_0-1} u\cdot \pr{\frac{\epsilon_0}{2c_0}u}^{\frac 1\sigma-1}.
    \end{align*}
    Taking $C_0 = C_0(n, \epsilon_0):= \sigma^{-1}\frac{nc_0-1}{2}(\frac{\epsilon_0}{2c_0})^{1-\frac{1}{\sigma}}$, this can be written as $\sigma\frac{1}{f} (\partial_t -\Delta) f \geq C_0^{-1} u^{\frac{1}{\sigma}}$. 
    Thus, \eqref{u-evol-1} becomes 
    \begin{align}\label{u-evol-2}
    (\bd_t-\D) u
    & \leq  
    2\pair{\n u, \n \log f^{1-\sigma}} 
    - C_0^{-1} u^{1+\frac 1\sigma}
    - \sigma(1-\sigma)u\frac{|\n f|^2}{f^2}.
    \end{align}

    Finally, following an idea in \cite{L24}, we estimate
    \begin{align*}
    2\pair{\n u, \n \log f^{1-\sigma}}
    \le (1-\sigma) \pr{
    \frac{|\n u|^2}{\sigma u}
    + \sigma u \frac{|\n f|^2}{f^2}
    },
    \end{align*}
    so \eqref{u-evol-2} implies
    \begin{align*}
    (\bd_t-\D) u
    & \le (1-\sigma) \pr{
    	\frac{|\n u|^2}{\sigma u}
    	+ \sigma u \frac{|\n f|^2}{f^2}
    }
    - C_0^{-1} u^{1+\frac 1\sigma}
    - \sigma(1-\sigma)u\frac{|\n f|^2}{f^2}\\
    & = \frac{1-\sigma}{\sigma} \frac{|\n u|^2}{u}
    - C_0^{-1} u^{1+\frac 1\sigma}.
    \end{align*}
    This is equivalent to
    \begin{align}\label{evol-of-u}
    (\bd_t-\D)u^{\frac{1}{\sigma}}
    & = \frac 1\sigma u^{\frac 1\sigma - 1}(\bd_t-\D) u
    - \frac 1\sigma\pr{\frac 1\sigma-1} u^{\frac 1\sigma-2} |\n u|^2\\
    & = \frac{1}{\sigma} u^{\frac{1}{\sigma}-1}\left((\partial_t - \Delta)u - \frac{1-\sigma}{\sigma} \frac{|\nabla u|^2}{u}\right)\nonumber\\
    & \leq -C_0^{-1} u^{\frac 2\sigma}.\nonumber
    \end{align}
    Since we assume that each $M_t$ has bounded curvature, we have that $B := \sup_{M \times [0, T]} u < \infty$. It follows that $t u^{\frac{1}{\sigma}}-C_0 \leq 0$ holds on $M \times [0, \tau]$ for $\tau = \frac{1}{2}C_0 B^{-\frac{1}{\sigma}}$. On the other hand, \eqref{evol-of-u} implies 
    \begin{align*}
    (\partial_t - \Delta) (tu^{\frac{1}{\sigma}}-C_0) \leq u^{\frac{1}{\sigma} } - t C_0^{-1} u^{\frac{2}{\sigma}} = -\frac{1}{C_0}(u^{\frac{1}{\sigma}})\left((tu^{\frac{1}{\sigma}})-C_0\right) \leq -\frac{B^{\frac{1}{\sigma}}}{C_0} \left(C_0 - t u^{\frac{1}{\sigma}}\right),
    \end{align*}
    and consequently $(\partial_t - \Delta) (e^{\tilde{B}t} (tu^{\frac{1}{\sigma}} - C_0)) \leq 0$ for $\tilde{B} = B^{\frac{1}{\sigma}}/C_0$. 
    The estimate $t u^{\frac{1}{\sigma}} \leq C_0$ on $M \times [0, T]$ now follows from the maximum principle for complete evolving manifolds with bounded curvature. See for instance Chapter 12 of \cite{CCG+08}. The estimate for $u$ implies $|\hat{A}|^2 \leq C_0^\sigma f^{1-\sigma}t^{-\sigma}$ on $M \times (0, T]$. Since $f \leq c_0 |H|^2$, this completes the proof of Theorem \ref{thm:BHS-type-est}.
\end{proof}

\subsection{\bf Integral estimate}
\label{sec:integral-form}

Using the estimate developed above, we can obtain an integral version of the planarity estimate.
It allows a larger class of ancient solutions and similarly implies rigidity results.
Suppose $F\colon M\times [0,T]\to \bb R^N$ is a complete solution of MCF.
On $M,$ we let $dV_t$ be the volume form of the metric $g_t$ induced from the Euclidean metric by the immersion $F(\cdot, t).$
In the following, we suppose that for every $t \in [0, T]$,
\begin{equation}\label{eq:volume-bounds}
\int_M \Phi_T \, dV_t \leq \Lambda,
\end{equation}
and 
\begin{equation}\label{eq:curvature-int-bounds}
    \int_M \pr{|A|^2 + t|\nabla^\perp A|^2 + t^2|(\nabla^\perp)^2 A|^2 + t^2|\partial_t^\perp A|^2 }\Phi_T \,dV_t <\infty,
\end{equation}
where $\Phi_T(x, t) := (4\pi(T-t))^{-\frac{n}{2}} \exp\pr{-\frac{|x|^2}{4(T-t)}}$.  Note that the high decay rate of the Gaussian a priori allows considerable growth of the curvature. However, we also now need to assume weighted area bounds as in \eqref{eq:volume-bounds}. Since $\int_M \Phi_T \, dV_t \leq \sup_{t \in [0, T]}\lambda(M_t)$, a uniform entropy bound suffices. 

\begin{thm}\label{thm:BHS-type-est-int-version}
	Suppose $n \geq 5$ and $N > n + 1$.
	Let $M^n_t = F(M^n, t) \subset \mathbb{R}^N$, $t\in[0,T]$, be a smooth family of $n$-dimensional, complete, connected, immersed submanifolds evolving by MCF. Suppose both \eqref{eq:volume-bounds} holds with constant $\Lambda > 0$ and  \eqref{eq:curvature-int-bounds} holds, for all $t \in [0, T]$. Finally, suppose there exists a constant $\epsilon_0 \in (0, c_n)$ such that $|A|^2 \leq (c_n-\epsilon_0)|H|^2$ holds on $M^n \times [0, T]$. Then either $M^n_t$ is a static, flat copy of $\mathbb{R}^n$ in $\mathbb{R}^N$, or else $|H| > 0$ on $M \times (0, T]$ and there exist constants $C=C(n,\epsilon_0)$ and $\sigma = \sigma(n,\epsilon_0)$ such that
	\begin{equation}\label{eq:int-estimate}
	\int_{M} \pr{\frac{|\hat A|^2}{f^{1-\sigma}}}^{1/\sigma} \Phi_T dV_t
	\le \frac{C\Lambda}{t}.
	\end{equation}
	holds for all $t\in (0, T]$. 
\end{thm}

\begin{proof}
As before, the strong maximum principle implies that if the flow is not static and flat, then we may assume $|H| > 0$ on $M \times (0, T]$. Proceeding under this assumption, let us define $c_0, f$, and $u:= |\hat A|^2/f^{1-\sigma}$ as in the proof of the pointwise estimate in Section~\ref{sec:pointwise-planarity}. 
Let us additionally take $C_0 := C_0(n, \epsilon_0)$ and $\sigma := \sigma(n, \epsilon_0)$ as chosen there. 
Consider the scale-invariant function $\tilde{u} := tu^{\frac{1}{\sigma}}$. Let us first verify integrability. 
Note that
\begin{align*}
  |\partial_t \tilde{u}|+   |\nabla^2 \tilde{u}| & \leq C\left( |A|^2 + t |\nabla^\perp A|^2 + t^2 |(\nabla^\perp)^2 A|^2 + t^2 |\partial_t^\perp A|^2\right)
\end{align*}
so that, by assumption \eqref{eq:curvature-int-bounds}, the quantity $(\partial_t - \Delta)\tilde{u}$ is integrable with respect to the Gaussian weight. From \eqref{evol-of-u} above, we have that 
\begin{equation}\label{eq:tilde-u}
(\partial_t -\Delta)\tilde{u} \leq \frac{1}{t} \tilde{u} - \frac{1}{C_0t} \tilde{u}^2.
\end{equation}
In particular, $\tilde{u}^2 \leq C_0 \tilde{u} - C_0 t(\partial_t -\Delta)\tilde{u} \leq \frac{1}{2} \tilde{u}^2 + \frac{1}{2} C_0^2 + C_0 T (|\partial_t \tilde{u}| + |\Delta \tilde{u}|)$ by Cauchy's inequality and the fact that $t \leq T$. This shows that $\int_M \tilde{u}^2 \, \Phi_T \, dV_t < \infty$, so that both $\tilde{u}^2$ and $\tilde{u}$ (by application of H\"older) are integrable with repsect to the Guassian weight as well. 

Based on \eqref{eq:tilde-u}, Ecker's local monotonicity formula \cite[Theorem~4.13]{E04}, and H\"older's inequality, we have 
\begin{align*}
\bd_t \int_{M} \tilde{u}\; \Phi_T dV_t
& = \int_{M} \pr{(\bd_t - \D) \tilde{u}
- \abs{H + \frac{x^\perp}{2(T-t)}}^2 \tilde{u}} \Phi_T dV_t\\
& \le \frac{1}{t} \int_M  \tilde{u} \;\Phi_T dV_t -\frac{1}{C_0t} \int_{M} \tilde{u}^2\; \Phi_T dV_t\\
& \le\frac{1}{t} \int_M  \tilde{u} \;\Phi_T dV_t -\frac{1}{C_0t}\pr{\int_{M} \Phi_T dV_t}^{-1} 
 \pr{\int_{M} \tilde{u} \;\Phi_T dV_t}^2.
\end{align*}
In view of the weighted area bounds $\eqref{eq:volume-bounds}$, we conclude
\begin{align*}
\bd_t \int_{M}\tilde{u} \;  \Phi_T dV_t
\le \frac{1}{t} \int_M  \tilde{u} \;\Phi_T dV_t -\frac{1}{C_0\Lambda t}
 \pr{\int_{M} \tilde{u} \;\Phi_T dV_t}^2.
\end{align*}
Letting $F(t):= \frac{1}{C_0\Lambda} \int_{M_t} \tilde{u}\, \Phi_T dV_t$, we have shown
\begin{align*}
F'(t) \le  \frac Ft \pr{1-F},
\end{align*}
so $F(t)\le 1$ (since $F(0) = 0$). 
That is, 
\begin{align*}
\int_{M} u^{\frac{1}{\sigma}} \,\Phi_T dV_t
\le \frac{C_0\Lambda}{t}.
\end{align*}
This finishes the proof.
\end{proof}

Theorem~\ref{thm:BHS-type-est-int-version} implies Corollary~\ref{cor:complete-bdd-curv-pinched}.
As mentioned, it also implies a new classification result for MCF shrinkers in any codimension, saying that a uniformly $c_n$-pinched shrinker with bounded entropy must be a generalized cylinder $S^{n-k}\pr{\sqrt{-2(n-k)}}\times \bb R^{k}.$
See Section~\ref{sec:corollary} for the proof of Corollary~\ref{cor:cyl-shrinker}.
As noted before, a generalized cylinder satisfies $|A|^2 = \frac{1}{n-k} |H|^2$ for $k \in \{0, \dots, n-1\}$ and so appears within the above class of shrinkers whenever $\frac{1}{n-k} < c_n$.

\section{\bf Convexity estimate}
\label{sec:conv}

\textit{On notation in Section \ref{sec:conv}:} Recall that for Theorem~\ref{thm:improved-conv}, we focus on a family $F\colon M\times (-\infty, 0)\to \bb R^{n+1}$ of hypersurfaces evolving by MCF. Since we work in codimension one, in this section we let $H$ denote the usual scalar mean curvature, which we assume is nonnegative throughout what follows. 
We let $\nu$ denote a choice of unit normal (we assume $M$ is orientable) and $\mathbf{H} = - H \nu$ the vector-valued mean curvature. To match standard conventions, we let $A$ denote the scalar-valued $(0,2)$-tensor second fundamental form through this section (so that $H = -\tr A = - g^{ij} A_{ij}$). The only connection we need is the usual one $\nabla$ here.

We will prove Theorem~\ref{thm:improved-conv} in this section.
We will use an integral estimate to prove a convexity estimate, which says that the smallest principal curvature $\lambda_1$ satisfies $\lambda_1\ge -\varepsilon H$ for all $\varepsilon\in(0,1),$ implying $\lambda_1\ge 0.$

The uniform pinching condition assumed in Theorem \ref{thm:improved-conv} says that there exists $L>0$ such that $|A|\le LH.$
We will make use of the pinching quantity $LH-|A|$ a lot.
In most of the parts of this section, we will assume $LH\ge |A|>0,$ since otherwise the strong maximum principle will imply that $H$ and $|A|$ both vanish identically.
We have seen this in Section~\ref{sec:Planarity}, and will go through it again at the beginning of the proof of Theorem~\ref{thm:improved-conv}.

For $\varepsilon\in(0,1)$ and $\sigma\in(0,1),$ we let $G_\varepsilon
:=\max\set{-\lambda_1 - \varepsilon (2LH-|A|),0}$ 
and 
\[
G:=G_{\varepsilon,\sigma}
:=\frac{G_\varepsilon}{H^{1-\sigma}},
\]
when $H> 0$. Our goal is to show that given $\varepsilon\in(0,1),$ we can suitably choose $\sigma\in(0,1)$ such that  $G_{\varepsilon,\sigma}\le 0,$ and hence $G_\varepsilon\le 0.$
This will then imply the estimate $\lambda_1\ge -\varepsilon H$ mentioned above. 
When $\varepsilon$ and $\sigma$ are fixed (which is true until the proof of Theorem~\ref{thm:improved-conv} given at the end of this section) we will just write $G$ instead of $G_{\varepsilon,\sigma}.$

The main reason why we can handle $L$-pinching for any $L$ is that our goal is to show $G_\varepsilon = 0.$
This means that we only care about points at which $G_\varepsilon>0$ a priori, which means
\begin{align}\label{lambda-pinched}
\lambda_1
< -\varepsilon\pr{2LH-|A|}
\le -\varepsilon LH
\le -\varepsilon |A|.
\end{align}
This gives us good lower bounds for both the evolution of the pinching quantity and the symmetrized curvature tensor that leads to the Poincar\'e-type inequality.

We start by deriving the basic evolution equation of $G_\varepsilon.$
Note that $G_\varepsilon$ may not be smooth, so the differential equation we get holds in the barrier sense.

\begin{lemma}
\label{lem:f_varepsilon-evol}
	Suppose $0<|A|\le LH$ on $M_t$ for some $L\ge 1.$
    Given $\varepsilon\in(0, \frac{1}{L}),$ there exists $\gamma_0 = \gamma_0(n, \varepsilon, L)>0$ such that 
    \begin{align}\label{f_varepsolon-evol}
    (\bd_t-\D) G_\varepsilon
    \le G_\varepsilon\pr{|A|^2 - \gamma_0\frac{|\n A|^2}{H^2}}
    \end{align}
    in the barrier sense.
\end{lemma}

\begin{proof}
	By \cite[Proposition~12.9]{ACGL} (cf. \cite{B15,LN22}), we have 
	\begin{align}\label{lambda1-evol}
	(\bd_t-\D) \lambda_1
	\ge \lambda_1 |A|^2
	\end{align}
	in the barrier sense.
	
    Next, we look at the evolution of $G_\varepsilon.$
    We first work on the region $\set{G_\varepsilon>0}$, where $G_\varepsilon = -\lambda_1-\varepsilon(2LH -|A|) > 0$.
    The evolution equations of $A$ and $H$ imply
    \begin{align}\label{A-LH-evol}
    	(\partial_t - \Delta)(|A|-2LH)=& \frac{(\partial_t - \Delta)|A|^2 + 2|\nabla |A||^2  }{2|A|} - 2L|A|^2H \\
    	=& \frac{ |A|^4  - |\nabla A|^2 + |\nabla|A||^2}{|A| }- 2L|A|^2H
        \nonumber\\
    	=& |A|^2(|A|-2LH) - \frac{|\nabla A|^2-|\nabla|A||^2 }{|A|}.\nonumber
    \end{align}
    Using $G_\varepsilon \leq -\lambda_1\le |A|\le LH$ and \eqref{lambda1-evol}, this implies
    \begin{align}\label{A-LH-evol-2}
    	(\partial_t - \Delta)G_{\varepsilon} 	
        \leq |A|^2G_{\varepsilon} - \frac{\varepsilon\pr{|\nabla A|^2 - |\nabla |A||^2}}{|A|}
    	\leq |A|^2G_{\varepsilon} - \frac{\varepsilon\pr{|\nabla A|^2 - |\nabla |A||^2}}{L^2 H^2} G_{\varepsilon}.
    \end{align}
    On $\set{G_\varepsilon>0},$ by \eqref{lambda-pinched}, we have $\lambda_1\le -\varepsilon LH,$ so by Lemma~\ref{Appendix-grad-lower-bound} (c.f \cite[Lemma~2.1]{L17}) and the pinching assumption, we have
    \begin{align*}
    	\frac{|\nabla A|^2 - |\nabla |A||^2}{H^2} \geq \frac{\varepsilon^2 L^2}{8n^2} \frac{|\nabla A|^2}{|A|^2}
        \ge \frac{\varepsilon^2 }{8n^2} \frac{|\nabla A|^2}{H^2}.
    \end{align*} 
    Plugging this into \eqref{A-LH-evol-2},
    we obtain
    \begin{align*}
    	(\partial_t - \Delta)G_{\varepsilon} 	\leq |A|^2G_{\varepsilon} - \frac{\varepsilon^3}{8n^2L^2} \frac{|\nabla A|^2}{H^2} G_{\varepsilon} .
    \end{align*}
    The desired estimate \eqref{f_varepsolon-evol} then holds with $\gamma_0 = {\varepsilon^3}/\pr{8n^2L^2}$ on $\set{G_\varepsilon>0}.$  

     To see that \eqref{f_varepsolon-evol} also holds in the barrier sense on the region $\{G_{\varepsilon} = 0\}$, it suffices to observe that since $G_{\varepsilon} \geq 0$, we can take the zero function as a lower barrier. Since, of course, $(\partial_t - \Delta) 0 = 0$, we conclude 
     \[
     (\partial_t -\Delta)G_{\varepsilon} \leq 0 = G_{\varepsilon} \left(|A|^2 - \gamma_0 \frac{|\nabla A|^2}{H^2}\right)
     \]
     holds at any point in $\{G_{\varepsilon} = 0\}$. This finishes the proof of the lemma.
\end{proof}

As a consequence, we derive an evolution equation for weighted integral of $G=G_{\varepsilon,\sigma}.$
Recall that we consider the weight
$$\Phi(x,t)
:= \frac 1{(-4\pi t)^{n/2}}
e^{\frac{|x|^2}{4t}}$$
given by the backward heat kernel.
As in Section~\ref{sec:integral-form}, for the family of immersions $F\colon M\times (-\infty, 0)\to \bb R^{n+1},$ on $M,$
we let $dV_t$ be the volume form of the metric $g_t$ induced from the Euclidean metric by the immersion $F(\cdot, t).$
We note that based on the Shi-type estimate of Ecker--Huisken~\cite{EH91} (cf. \cite[Proposition~3.22]{E04}), the polynomial growth assumption on $|A|$ implies that $|\n A|$ is also of polynomial growth on each time slice. 
In particular, given each $t<0,$ both $|A|\Phi(\cdot,t)$ and $|\n A|\Phi(\cdot,t)$ decay exponentially fast.

\begin{lemma}
\label{lem:A-poly-imply-nA-poly}
	If $M_t$ is an ancient MCF such that $|A|(\cdot,t)$ has polynomial growth for each $t,$ then for all $m\ge 0,$ $|\n^m A|(\cdot,t)$ also has polynomial growth for each $t.$
\end{lemma}

\begin{proof}
    We prove the lemma when $m=1,$ and as in the Ecker--Huisken estimate, the case when $m>1$ can be done by induction.
	We use the form in \cite[Proposition~3.22]{E04}, whose proof says that if $|A|\le c_0$ in $ B_2(x_0)\times (t_0-4,t_0),$ then there exists a universal polynomial $\mathcal P$ such that $|\n A|\le \mathcal P(c_0, n)$ in $ B_1(x_0)\times (t_0-1,t_0).$  
	See \cite[Page 46]{E04}.
	We use this to prove the lemma.
	
	We fix $t_0<0.$
	By the assumption, there exists a polynomial $p$ such that $|A|(x,t)\le p(|x|)$ for $t\in (t_0-4, t_0)$ and $x\in M_t$ with $|x|\ge 1.$
	We may assume all the coefficients of $p$ are non-negative, so in particular, given $s_0\ge 1,$ $p(s)\le p(2s_0)$ for all $s\in (s_0-1, s_0+1).$
	The Ecker--Huisken estimate above then implies 
	\begin{align*}
	|\n A(x,t)|
	\le \mathcal P\pr{\sup_{B_1(x)\times (t-1,t)}|A|, n}
	\le \mathcal P\pr{\sup_{s\in (|x|-1, |x|+1)}p(s), n}
	\le \mathcal P \pr{p\pr{2|x|}, n}
	\end{align*}
	for $t\in (t_0-1,t_0)$ and $x\in M_t.$
	This is also a polynomial in $|x|$, and the lemma follows.
\end{proof}

\begin{remark}
	\label{rmk:A-poly-high-codim}
	Although it is not used in this section, Lemma~\ref{lem:A-poly-imply-nA-poly} is also true for higher codimension MCF.
	This can be derived based on the Shi-type estimate proven by Andrews--Baker in~\cite[Lemma~3]{AB}.
\end{remark}

As mentioned, Lemma~\ref{lem:A-poly-imply-nA-poly}, implies that for all $m\ge 0,$ $|\n^m A|\Phi$ has exponential decay.
In particular, based on the bounded entropy assumption, we can freely take the weighted integrals of arbitrary powers of the curvature and its derivatives.

\begin{cor}
\label{cor:evol-of-int-Gp}
	Suppose $0<|A|\le LH$ on $M_t$ for some $L\ge 1.$
    Given $\sigma\in (0,1)$ and $\varepsilon\in (0,\frac{1}{L}),$ there exists $\gamma = \gamma(\varepsilon, L)>0$ such that for any $p \geq 2 + \frac{2n}{\gamma}$,
    letting $v = G^{p/2}$, the estimate
    \begin{align}\label{fp-Phi-evol}
    \bd_t \int_{M}
    v^2\Phi dV_t
    \le \int_{M}
    \pr{
    \sigma p v^2|A|^2
	- \gamma p  v^2\frac{|\n A|^2}{H^2}
	- 2|\n v|^2
    - v^2 \abs{\mathbf{H} + \frac{x^\perp}{-2t}}^2
    }\Phi dV_t
    \end{align} 
    holds for almost every $t \in (-\infty, 0)$.
\end{cor}

\begin{proof}
	By Lemma~\ref{lem:f_varepsilon-evol}, we have
	\begin{align*}
	&\quad (\bd_t-\D) G\\
	& = \frac 1{H^{1-\sigma}} (\bd_t-\D) G_\varepsilon
	- \frac{G_\varepsilon}{H^{2-2\sigma}}
	(\bd_t-\D)H^{1-\sigma}
	+ \frac 2{H^{1-\sigma}} \pair{\n H^{1-\sigma}, \n G}\\
	& \le \frac 1{H^{1-\sigma}}\cdot
	G_\varepsilon\pr{|A|^2 - \gamma_0\frac{|\n A|^2}{H^2}}
	- \frac{G_\varepsilon}{H^{2-2\sigma}}
	\pr{
	\frac{1-\sigma}{H^\sigma} \cdot H|A|^2
	+ \sigma(1-\sigma) \frac{|\n H|^2}{H^{1+\sigma}}
	}
	+ \frac 2{H^{1-\sigma}} \pair{\n H^{1-\sigma}, \n G}\\
	& = \sigma G |A|^2
	- \gamma_0 G \frac{|\n A|^2}{H^2}
	+ \frac{2(1-\sigma)}H \pair{\n H, \n G} 
    - \sigma(1-\sigma)G \frac{|\nabla H|^2}{H^2}.
	\end{align*}
	Throwing away the last term above, and using Young's inequality on the third term, we find 
    \begin{align*}
        (\partial_t - \Delta) G  &\leq \sigma G |A|^2 - \gamma_0 G \frac{|\nabla A|^2}{H^2} + \frac{\gamma_0}{2n}G \frac{|\nabla H|^2}{H^2} + \frac{2n}{\gamma_0} G \frac{|\nabla G|^2}{G^2} \\
        & \leq \sigma G |A|^2 - \frac{1}{2}\gamma_0 G \frac{|\nabla A|^2}{H^2} + \frac{2n}{\gamma_0} G \frac{|\nabla G|^2}{G^2}
    \end{align*}
    where we use the rough estimate $|\n H|^2 \le n|\n A|^2$. Thus, 
    \begin{align*}
    (\partial_t - \Delta) G^p &= p G^{p-1} (\partial_t - \Delta) G - p(p-1) G^{p-2} |\nabla G|^2\\
    & \leq \sigma p G^p |A|^2 - \frac{1}{2} \gamma_0 p G^p \frac{|\nabla A|^2}{H^2} + \frac{2np}{\gamma_0} G^{p-2}|\nabla G|^2 - p(p-1) G^{p-2} |\nabla G|^2\\
    & =  \sigma p G^p |A|^2 - \frac{1}{2} \gamma_0 p G^p \frac{|\nabla A|^2}{H^2} - 2|\nabla v|^2+ \left(- \frac{p^2}{2} + p + \frac{2n}{\gamma_0} p\right)G^{p-2}|\nabla G|^2
    \end{align*}
    where we use $2|\nabla v|^2 = \frac{p^2}{2} G^{p-2}|\n G|^2,$ based on the definition $v^2= G^p.$
	Thus, the pointwise estimate follows (after taking $\gamma = \frac{1}{2}\gamma_0$) if 
    \begin{align*}
	-\frac{p^2}{2} + p + \frac{n}{\gamma}p\le 0,
	\end{align*}
    which is true if $p\ge 2+\frac{2n}\gamma.$
    Assuming this, we thus have
    \begin{align*}
	(\bd_t-\D) G^p
	& \le \sigma p G^p|A|^2
	- \gamma p  G^p\frac{|\n A|^2}{H^2}
	- 2|\n v|^2.
	\end{align*}
	
    With this differential inequality, the integral estimate then follows.
	In fact, since the second fundamental form $A$ is smooth and $\lambda_1$ is the smallest eigenvalue of $-A,$ we know that $G_\varepsilon$ (and hence $G$) is a locally Lipschitz function and can locally be expressed as the sum of a smooth function and a concave function (cf. \cite{B15, LN22}).
	Alexandrov's theorem~\cite{A39} implies that $G$ admits a second order Taylor expansion away from a set of measure zero.
	In particular, \eqref{f_varepsolon-evol} can be understood in the distributional sense (cf. \cite[Section~6.4]{EG92}).
	Thus, given a non-negative smooth function $\eta\colon M\times(-\infty,0)\to\bb R$ with compact support, we have
	\begin{align}\label{Gp-distribution}
	\bd_t \int_{M}
	G^p\eta dV_t
	\le \int_{M}
	\pr{
		\sigma p G^p|A|^2
		- \gamma p  G^p\frac{|\n A|^2}{H^2}
		- 2|\n v|^2
	}\eta dV_t
	+ \int_M G^p \pr{\bd_t+\D -H^2} \eta dV_t
	\end{align}
	for almost every $t.$
	
    We now use a standard cutoff trick to prove \eqref{fp-Phi-evol}, cf. \cite[Section~4]{E04}.
	Let $\varphi_R\colon \bb R^{n+1}\to\bb R$ be a smooth function such that $0\le \varphi_R\le 1,$ $\varphi|_{B_R}=1,$ and $\varphi|_{\bb R^{n+1}\setminus B_{2R}}=0.$
	We can find such a $\varphi_R$ for all large $R$ such that
	\begin{align*}
	R\abs{\n^{\rm Euc} \varphi_R}
	+ R^2\abs{\Hess^{\rm Euc}_{\varphi_R}} \le C_n
	\end{align*}
	for some $C_n<\infty.$
	Thus, if we abuse the notation and 
    write $\varphi_R\pr{p,t}
	= \varphi_R\pr{F(p,t)},$ we have
	\begin{align*}
	\abs{(\bd_t - \D)\varphi_R}
	= \abs{
	\pr{- \D^{\rm Euc}}\varphi_R
	+ \Hess^{\rm Euc}_{\varphi_R}(\nu, \nu)
	}
	\le \frac{C_n}{R^2} \chi_{B_{2R}\setminus B_R}
	\end{align*}
	where $\nu$ is a fixed unit normal of $F$ and $\chi_{B_{2R}\setminus B_R}$ is the characteristic function of $B_{2R}\setminus B_R.$
    Observe that 
    \begin{align*}
    \pr{\bd_t+\D -H^2} \pr{\Phi \varphi_R}  &= \varphi_R\pr{\bd_t+\D -H^2} \Phi + \Phi (\partial_t + \Delta) \varphi_R + 2 \langle \nabla \Phi, \nabla \varphi_R\rangle\\
    & = \varphi_R\pr{\bd_t+\D -H^2} \Phi + \Phi (\partial_t - \Delta) \varphi_R + 2 \mathrm{div}(\Phi \nabla \varphi_R),
    \end{align*}
    and recall the evolution equation $(\bd_t +\D -|\h |^2)\Phi = -\abs{\h + \frac{x^\perp}{-2t}}^2\Phi$.
	Thus, plugging in $\eta = \Phi\varphi_R$ to \eqref{Gp-distribution}, we can estimate
	\begin{align*}
	\bd_t \int_{M}
	G^p\Phi \varphi_R dV_t
	& \le \int_{M}
	\pr{
		\sigma p G^p|A|^2
		- \gamma p  G^p\frac{|\n A|^2}{H^2}
		- 2|\n v|^2
	}\Phi \varphi_R dV_t
	\\
    & \qquad + \int_M G^p \pr{\bd_t+\D -H^2} \pr{\Phi \varphi_R} \; dV_t\\
	& = \int_{M}
	\pr{
		\sigma p G^p|A|^2
		- \gamma p  G^p\frac{|\n A|^2}{H^2}
		- 2|\n v|^2
	}\Phi \varphi_R \;dV_t \\
	&\quad + \int_M G^p \pr{\bd_t + \D - H^2}\Phi \cdot \varphi_R dV_t+ \int_M G^p \Phi\cdot (\bd_t-\D)\varphi_R \; dV_t\\
	& \le \int_{M}
	\pr{
		\sigma p G^p|A|^2
		- \gamma p  G^p\frac{|\n A|^2}{H^2}
		- 2|\n v|^2
		- G^p \abs{\h + \frac{x^\perp}{-2t}}^2
	}\Phi \varphi_R dV_t\\
	&\quad + \frac{C_n}{R^2} \int_M \chi_{B_{2R}\setminus B_R} \Phi dV_t
	\end{align*}
	for almost every $t$.
	The boundedness of all the integrands then allows us to use the dominated convergence theorem to conclude the proof.
\end{proof}

To deal with the large reaction term in the evolution equation in \eqref{fp-Phi-evol}, we need a Poincar\'e-type inequality.
It essentially makes use of the pinching condition to show that the curvature can be bounded by its derivatives (cf. \cite[Proposition 3.2]{LN22}).
The techniques have been developed in many previous works \cite{H84, HS99, AB, L17, LN22} as an essential step of the Stampacchia iteration (cf. Remark~\ref{rmk:MCF-evol}).
The main difference here is that we are looking at the weighted integral, so an additional term involving the mean curvature will appear. 
It will finally be absorbed by the shrinker quantity term in \eqref{fp-Phi-evol}. In the following, we say $u  \in H^2_W := H^2_W(M)$ at time $t < 0$ if 
\[
\int_M (u^2 + |\nabla u|^2 + |\nabla^2 u|^2) \Phi \, dV_t < \infty.
\]
where as above $\Phi(x, t) = (-4\pi t)^{-n/2} e^{|x|^2/4t}$. The next result is applied at a fixed time $t < 0$.

\begin{prop}
\label{prop:Poincare}
    Suppose $0<|A|\le LH$ on $M_t$ for some $L>0$ for some $t < 0$.
    Given $\delta>0$ and $\ovl\varepsilon>0,$ there exists $P_\delta <\infty$ depending only on $\delta,$ $\ovl \varepsilon,$ and $L$ such that the following holds.
    If an $H^2_W$ function $u\colon M\to\bb R$ is supported in $\set{\lambda_1\le -\ovl\varepsilon H}$ such that both $u$ and $\n u$ are of polynomial growth, then it satisfies 
	\begin{align*}
	\int_M u^2 |A|^2\Phi dV_t
	\le \int_M\pr{
	\delta |\n u|^2
	+ P_\delta u^2 \frac{|\n A|^2}{H^2}
	+ \frac{\delta}{|t|} u^2 \pr{1+|\pair{\h, x^\perp}|}
	}\Phi dV_t.
	\end{align*}
\end{prop}

\begin{proof}
    Other than through the weight $\Phi$, the result doesn't depend upon the time $t$.  So in what follows, let us write $dV = dV_t$ and suppress the time dependence. 
    
    We let $C:=A\otimes A^2 - A^2 \otimes A$ 
    and let $\lambda_i$'s be the principal curvatures at a point $p\in M$ with $\lambda_1\leq \cdots\leq \lambda_n$. 
    At $p,$ if we choose geodesic normal coordinates with respect to which $A$ is diagonalized, we can calculate, at $p,$
    \begin{align*}
	|C|^2 = \sum_{i,j=1}^n \big( A_{ii}A_{jj}^2 - A_{jj}A_{ii}^2 \big)^2 
	= \sum_{i,j=1}^n \big( \lambda_i\lambda_j^2 - \lambda_j\lambda_i^2 \big)^2
	=& \sum_{i,j=1}^n \lambda_i^2\lambda_j^2(\lambda_j-\lambda_i)^2.
    \end{align*}
    If, moreover, $p$ lies in $\set{\lambda_1\le -\ovl\varepsilon H},$ then we have, at $p,$
    \begin{align}\label{|C|-lower-bound}
    |C|^2
    \geq \lambda_n^2\lambda_1^2(\lambda_n-\lambda_1)^2
    \geq \frac{1}{n^2}H^2
    \cdot \ovl\varepsilon^2 H^2
    \cdot \pr{\frac{1}{n}+\ovl\varepsilon}^2H^2 
    \geq \alpha|A|^2H^4
    \end{align}    
    where $\alpha:=\ovl\varepsilon^2/(L^2n^4)$ and we use $\lambda_n \geq H/n$ (see \eqref{lambda_n-lower-bound}).

	On the other hand, Simons' identity implies 
	\begin{align}\label{C-Simons}
	C_{ijk\ell}
	= \frac 12\pr{
	\n_i\n_j A_{k\ell}
	+ \n_j\n_i A_{k\ell}
	- \n_k \n_\ell A_{ij}
	- \n_\ell \n_k A_{ij}
	}.
	\end{align}
    Combining \eqref{|C|-lower-bound} and \eqref{C-Simons}, we get
	\begin{align}\label{Poin-Simons}
	\alpha
	\int_{M} u^2 |A|^2 \Phi dV
	& \le \int_M \frac{u^2}{H^4} C^{ijk\ell}
	\pr{\n_i\n_j A_{k\ell}
	-\n_k \n_\ell A_{ij}}
	\Phi dV,
	\end{align}
    where we use the assumption that the support of $u$ is contained in $\set{\lambda_1\le -\ovl\varepsilon H}.$
	We then integrate by parts to get
	\begin{align}\label{Poin-int-by-parts}
	&\quad \int_M \frac{u^2}{H^4} C^{ijk\ell}
	\n_i\n_j A_{k\ell} \Phi dV\\
	& = -\int_M
	\pr{\frac{2u\n_i u}{H^4} C^{ijk\ell}
	- \frac{4u^2\n_i H}{H^5}C^{ijk\ell}
	+ \frac{u^2}{H^4} \n_i C^{ijk\ell}
	+ \frac{u^2}{H^4} C^{ijk\ell}\cdot \frac{\n_i|x|^2}{4t}
	}
	\n_j A_{k\ell}\Phi dV
	\nonumber\\
	& \le c\int_M 
	\pr{
	|u\n u| \frac{|\n A|}{H}
	+ u^2 \frac{|\n A|^2}{H^2}
	+ u^2 \frac{|\n A|}{H}\cdot \frac{\abs{x^T}}{4|t|}
	}\Phi dV,\nonumber
	\end{align}
	where we use $|C|\le cH^3,$
	$|\n C|\le cH^2|\n A|,$ and 
	$\n \Phi = \Phi\cdot \frac{\n |x|^2}{4t}$
    for some $c=c(L,n)<\infty.$
    On the other hand, we can obtain a Sobolev-type inequality by looking at the divergence (cf. \cite{E00} and~\cite[Lemma~B.1]{BW17})
	\begin{align}\label{u^2-Sobolev}
	\mathrm{div}\pr{u^2 \Phi \cdot x^T} - u^2\langle \h, x^\perp \rangle \Phi = {\rm div}\pr{
	u^2\Phi\cdot x
	}
	= \pr{nu^2
	+ 2\pair{x^T,\n u}u
	+ u^2\cdot \frac{|x^T|^2}{2t}
	}\Phi,
	\end{align}
    In fact, we claim that 
        \begin{align}\label{int-by-part-u^2}
            \int_M {\rm div}\pr{
	u^2\Phi\cdot x^T
	}dV = 0.
        \end{align}
    To see this, we use the same cutoff function as in Corollary~\ref{cor:evol-of-int-Gp}.
    Let $\varphi_R\colon \bb R^{n+1}\to\bb R$ be a smooth function such that $0\le \varphi_R\le 1,$ $\varphi|_{B_R}=1,$ $\varphi|_{\bb R^{n+1}\setminus B_{2R}}=0,$ and $|\n^{\rm Euc}\varphi_R|\le C_n/R$
	for some $C_n<\infty.$
    Integrating by part then gives
    \begin{align*}
    \abs{\int_M {\rm div}\pr{
	u^2\Phi\cdot x^T 
	} \varphi_R dV}
    = \abs{\int_M- u^2\Phi \langle x^T,\nabla \varphi_R\rangle dV}
    \leq  2C_n\Big| \int_{M\backslash B_R} u^2\Phi  dV \Big|.
    \end{align*}
    Using the fact that  $u$ and $\n u$ have polynomial growth, the claim~\eqref{int-by-part-u^2} follows from taking $R\rightarrow\infty$ and applying the dominated convergence theorem.
    
    Combining \eqref{u^2-Sobolev} and \eqref{int-by-part-u^2}, we get 
	\begin{align*}
	\int_M u^2\cdot \frac{|x^T|^2}{2|t|} \Phi dV
	&\le \int_M \pr{
	nu^2 
	+ 4|t|\cdot|\n u|^2
	+ u^2\cdot \frac{|x^T|^2}{4|t|}
	}\Phi dV
    + \int_M u^2 |\pair{\h, x^\perp}| \Phi dV.
	\end{align*}
	After rearrangement, it implies
	\begin{align*}
	\frac 1{16}
	\int_M u^2\cdot \frac{|x^T|^2}{|t|^2} \Phi dV
	\le \int_M \pr{
	\frac{n}{4|t|} u^2
	+ |\n u|^2
	}\Phi dV
    + \frac 1{4|t|} \int_M u^2 |\pair{\h, x^\perp}| \Phi dV.
	\end{align*}
    Thus, given any $b>0,$
	\begin{align*}
	\int_M u^2 \frac{|\n A|}{H}\cdot \frac{|x^T|}{4|t|} \Phi dV
	&\le \int_M \pr{
	bu^2\cdot \frac{|x^T|^2}{16|t|^2}
	+ \frac 1{4b}u^2 \frac{|\n A|^2}{H^2}
    }\Phi dV
    \\
    &\le \int_M \pr{
	\frac{bn}{4|t|}u^2
	+ b |\n u|^2
	+ \frac{b}{4|t|} u^2 |\pair{\h, x^\perp}|
	+ \frac 1{4b}u^2 \frac{|\n A|^2}{H^2}
    }\Phi dV.
	\end{align*}
	Putting this into \eqref{Poin-int-by-parts}, we get that for any $a>0,$
	\begin{align*}
	&\quad \int_M \frac{u^2}{H^4} C^{ijk\ell}
	\n_i\n_j A_{k\ell} \Phi dV\\
	& \le c\int_M 
	\pr{
		|u\n u| \frac{|\n A|}{H}
		+ u^2 \frac{|\n A|^2}{H^2}
		+ u^2 \frac{|\n A|}{H}\cdot \frac{\abs{x^T}}{4|t|}
	}\Phi dV \\
    & \le c\int_M\pr{
	\pr{a + b}|\n u|^2
	+ \pr{\frac 1{4a} + 1 + \frac 1{4b}} u^2 \frac{|\n A|^2}{H^2} 
	+  \frac{bn}{4|t|}u^2
	+ \frac{b}{4|t|} u^2 |\pair{\h, x^\perp}|
	}\Phi dV.
	\end{align*}

	The same estimate holds for the other term in the right hand side of \eqref{Poin-Simons}. Recalling our choice of $\alpha$, we thus have 
    \[
    \int_M u^2 |A|^2 \Phi dV \leq c(n, L, \bar{\varepsilon}) \int_M (a+b) |\nabla u|^2 + \left(\frac{1}{4a} + 1 + \frac{1}{4b}\right) u^2 \frac{|\nabla A|^2}{H^2} + \frac{bn}{4|t|} u^2\Big(1 + |\langle \mathbf{H}, x^\perp \rangle|^2\Big) \Phi dV.
    \]
    The result now follows by taking $a, b > 0$ sufficiently small depending upon $c(n, L, \bar{\varepsilon})$ and $\delta$.
\end{proof}

We can now put the estimates together to complete the proof of Theorem~\ref{thm:improved-conv}.

\begin{proof}
[Proof of Theorem~\ref{thm:improved-conv}]
    As in Theorem~\ref{thm:BHS-type-est}, we first deal with the case when $H$ vanishes at some point.
    In fact, the evolution equation \eqref{A-LH-evol} particularly implies
    \begin{align*}
    (\bd_t-\D)
    \pr{2LH-|A|}
    \ge 0.
    \end{align*}
    Thus, if $H(p_0, t_0)=0$ for some $(p_0,t_0)\in M\times (-\infty,0),$ then the pinching condition implies $|A|(p_0, t_0)=0$ and hence $\pr{2LH-|A|}(p_0, t_0)=0.$
    The strong maximum principle then implies $\pr{2LH-|A|}$ vanishes identically, which implies $A$ vanishes identically.
    In this case, the flow consists of static hyperplanes, and the conclusion follows.

    Thus, in the rest of the proof, we may assume $H$ and $|A|$ are both positive everywhere.
    Fix $\varepsilon\in(0, 1/L),$ and for any $\sigma\in (0,1),$ consider $G=G_{\varepsilon,\sigma}.$
    Note that when $G>0,$ we have $G_\varepsilon >0,$ so $\lambda_1\le -\varepsilon\pr{2LH - |A|}\le -\varepsilon LH.$ 
    Moreover, $G$ is of polynomial growth, and hence lies in $H^2_W$ for all~$t.$
    Thus, we can apply Proposition~\ref{prop:Poincare} with $u = v =  G^{\frac{p}{2}}$ to get
    \begin{align}\label{fp-Phi-poincare}
    \int_{M} G^p |A|^2 \Phi dV_t
    \le \int_{M}
    \pr{
    	\delta|\n v|^2
    	+ P_\delta G^p\frac{|\n A|^2}{H^2}
        + \frac{\delta}{|t|} G^p\pr{1+ |\pair{\h, x^\perp}|}
    }\Phi dV_t.
    \end{align}
    Taking $\delta=1/8$ in \eqref{fp-Phi-poincare} and plugging it into \eqref{fp-Phi-evol} leads to (remembering $G^p=v^2$)
    \begin{align*}
    \bd_t \int_{M}
    G^p \Phi dV_t
    &\le\int_{M} \pr{
    G^p\sigma p |A|^2
    - 2|\n v|^2
    - \gamma p v^2\frac{|\n A|^2}{H^2}
    - G^p\abs{\h + \frac{x^\perp}{-2t}}^2
    }
    \Phi dV_t\\
    &\quad +\pr{ 
    -\int_{M} v^2 |A|^2\Phi dV
    + \int_{M}\pr{
    	\frac 18|\n v|^2
    	+ P_{1/8} v^2 \frac{|\n A|^2}{H^2}
    	+ \frac{1}{8|t|} v^2
    	+ \frac 1{8|t|} v^2 \abs{\pair{\h, x^\perp}}
    }\Phi dV
    }\\
    & \le \int_{M} \pr{
    \pr{\sigma p -1}v^2|A|^2
    +\pr{P_{1/8}-\gamma p} v^2 \frac{|\n A|^2}{H^2}
    + \frac 1{8|t|} v^2
    }\Phi dV,
    \end{align*}
    for almost every $t < 0$, where we have used that $\abs{\h + \frac{x^\perp}{-2t}}^2
    \ge 4|\pair{\h, \frac{x^\perp}{-2t}}|.$
    Take $p$ so large that $\gamma p>P_{1/8}.$
    After fixing such $p,$ take $\sigma$ so small that $\sigma p-1=-1/2.$
    That is, $\sigma p=1/2$ and $\sigma = 1/(2p).$
    We then have
    \begin{align}\label{fp-Phi-improved-estimate}
    \quad\bd_t \int_{M}
    G^p\Phi dV_t
    & \le -\frac 12\int_{M} G^p|A|^2 \Phi dV_t
    + \frac 18\int_{M} \frac 1{|t|} G^p
    \Phi dV_t,
    \end{align}
    for almost every $t < 0$.

    We will use this to derive a differential inequality and use it to get the convexity estimate.
    First, using $G_\varepsilon\le |A|\le LH,$ we have
    \begin{align*}
    G^{p\pr{1+\frac 2{\sigma p}}}
    = G^p\cdot \pr{\frac{G_\varepsilon}{H^{1-\sigma}}}^{2/\sigma}
    \le  
    L^{\frac{2}{\sigma}} G^p H^2 \leq n L^{\frac{2}{\sigma}}G^p |A|^2
    \end{align*}
    Then using this, the entropy bound $\Lambda:= \sup_{t < 0}\lambda(M_t)  < \infty$, and H\"older's inequality, we derive
    \begin{align}\label{fp-Phi-Holder}
    \int_{M} G^p\Phi dV_t
    & \le \pr{\int_{M}\Phi dV_t}^{\frac{2}{\sigma p + 2}}
    \pr{\int_{M} G^{p\pr{1+\frac 2{\sigma p}}} \Phi dV_t}^{\frac{\sigma p}{\sigma p+2}}\\
    & \le 
    \left(\Lambda^{\frac{2}{\sigma p}} nL^{\frac{2}{\sigma}}\right)^{\frac{\sigma p}{\sigma p+2}}
    \pr{\int_{M} G^p |A|^2 \Phi dV_t}^{\frac{\sigma p}{\sigma p+2}}.\nonumber
    \end{align}
    Combining \eqref{fp-Phi-improved-estimate} and \eqref{fp-Phi-Holder} with $\sigma p = 1/2,$ we have
    \begin{align}\label{fp-Phi-ODE}
    \bd_t \int_{M}
    G^p\Phi dV_t
    & \le -C_1 \pr{\int_{M} G^p\Phi dV_t}^5
    + \frac{1}{-8t} \int_{M}
    G^p\Phi dV_t
    \end{align}
    where 
    $C_1 := 1/\pr{2n\Lambda^4 L^{4p}}.$
    
    Let 
    \begin{align*}
    J(t) 
    := (-t)^{-\frac{1}{8}}\int_{M}
    G^p\Phi dV_t.
    \end{align*}
    We claim that $J(t) = 0$ for all $t < 0$.
    Note that \eqref{fp-Phi-ODE}  implies, for almost every $t < 0$,
    \begin{align*}
        J'(t) \leq \frac{1}{8t} J(t) -C_1(-t)^{-\frac{1}{8}} \pr{\int_{M} G^p\Phi dV_t}^5
    + \frac{1}{-8t}J(t) = -C_1(-t)^{\frac{1}{2}} \pr{(-t)^{-\frac{1}{8}}\int_{M} G^p\Phi dV_t}^5,
    \end{align*}
    which means
    \begin{align}\label{J-ODE-main}
    J' \leq -C_1(-t)^{\frac{1}{2}} J^5.
    \end{align}
    Suppose, for sake of contradiction that $J_0:= J(t_0) > 0$ for some $t_0 < 0$. 
    Then for almost every $t \leq t_0$, we have 
    \[
    J'(t) \leq -C_1(-t)^{\frac{1}{2}} J^5 \leq -C_1|t_0|^{\frac{1}{2}} J(t)^5,
    \]
    which implies 
    \[
    (J(t)^{-4})' \geq 4C_0 |t_0|^{\frac{1}{2}}
    \]
    for almost every $t \leq t_0$. Integrating both sides of the differential inequality and using the absolute continuity of $J$,  we have for all $t \leq t_0$
    \[
    0 \leq J(t)^{-4} \leq J(t_0)^{-4} -4C_0 |t_0|^{\frac{1}{2}} (t_0 - t).
    \]
    Sending $t \to -\infty$ implies $J(t_0)^{-1}$ is unbounded, contradicting our assumption that $J(t_0) > 0$. Thus $J(t) = 0$ for all $t < 0$.

    We can now finish the proof of the theorem.
    Since $J=0,$ we have
    $\lambda_1 \ge -\varepsilon(2LH-|A|).$
    Since this pointwise estimate is true for any $\varepsilon\in (0,1/L),$ it implies that $\lambda_1\ge 0.$
    This completes the proof.
\end{proof}


\section{\bf Consequences of planarity and convexity estimates}
\label{sec:corollary}

We provide the proofs of Corollaries~\ref{cor:BC-ADS-HS-LN} and~\ref{cor:cyl-shrinker} in this section.

The proof of Corollary~\ref{cor:BC-ADS-HS-LN} is based on a recent work of Lee--Topping~\cite{LT22-PIC}, which could be used to generalize a compactness criterion in \cite{LN21}.
We state it as follows.

\begin{prop}
	\label{prop:n-1-pinched-compact}
    Let $n\ge 3$ and $M$ be an $n$-dimeniosnal immersed submanifold in $\bb R^N$ with 
	\begin{align*}
	|A|^2 \le \pr{\frac 1{n-1}-\varepsilon_0}|H|^2
	\end{align*}
	for some $\varepsilon_0>0.$
	Then $M$ is either flat or compact.
\end{prop}

This proposition removes the bounded  curvature assumption in \cite[Corollary~3.3]{LN21}.
Therefore, as mentioned in \cite{LN21}, Proposition~\ref{prop:n-1-pinched-compact} can be viewed as the high codimensional version of Hamilton's theorem for pinched hypersurfaces~\cite{Ham94}.
The proof technique is the same as that of \cite[Corollary~3.3]{LN21}, but now we use a theorem of Lee--Topping, which is a generalization of the result by Ni--Wu~\cite{NW} used in \cite{LN21}.

\begin{proof}
	[Proof of Proposition \ref{prop:n-1-pinched-compact}]
	Based on the pinching condition, \cite[Lemma~3.2]{LN21} implies
	\begin{align}\label{PCO-pinched}
	\mathcal{R}\ge \frac{\varepsilon_0}2 |H|^2 \cdot {\rm Id} \geq 0,
	\end{align}
	where $\mathcal{R}$ is the curvature operator of the induced metric $g$ on $M$ and $\rm Id$ is the identity map on $\Lambda^2$. 
    In particular, the induced metric $g$ on $M$ has non-negative complex sectional curvature. 
    
    By the Gauss equation, if we let $S$ be the scalar curvature of $g,$ then \eqref{PCO-pinched} implies
	\begin{align*}
	\mathcal{R} - \frac{\varepsilon_0}2S\cdot {\rm Id}
	\ge \frac{\varepsilon_0}2 |H|^2 {\rm Id}
	- \frac{\varepsilon_0}2\pr{|H|^2-|A|^2}\cdot {\rm Id}
	= \frac{\varepsilon_0}2 |A|^2\cdot {\rm Id}.
	\end{align*}
    Now, we use $R$ to denote the curvature tensor of $g$ and $I$ to denote the identity curvature tensor, that is, $I_{ijk\ell} = \delta_{ik}\delta_{j\ell}
	- \delta_{i\ell}\delta_{jk}.$ 
    Then
	\begin{align*}
	R - \frac{\varepsilon_0}2S\cdot I
    \in {\rm C}_{\rm PIC1}
	\end{align*}
    where ${\rm C}_{\rm PIC1}$ is the PIC1 cone.
	Therefore, \cite[Theorem~1.2]{LT22-PIC} implies that $M_t$ is either flat or compact.
\end{proof}

As mentioned, Proposition~\ref{prop:n-1-pinched-compact} doesn't have any curvature restriction.
Thus, we can apply it to the setting of Corollary~\ref{cor:BC-ADS-HS-LN}.

\begin{proof}
[Proof of Corollary~\ref{cor:BC-ADS-HS-LN}]
	For (1), since the uniformly $\frac 1{n-2}$-pinched condition implies uniform two-convexity, the result directly follows from Corollary~\ref{cor:complete-bdd-curv-pinched}, Corollary~\ref{cor:convex-then-noncollapsing}, and the main results in \cite{BC21, ADS20}.

	For (2), by Proposition~\ref{prop:n-1-pinched-compact}, $M_t$ is either compact or flat.
	If $M_t$ is compact, then the result follows from \cite[Theorem~1.1]{LN21}.
	If $M_t$ is flat, then it is identically flat for all $t.$

    For (3), when $n\in\set{3,4},$ using
    \begin{align*}
    \frac{4}{3n}\le\frac 1{n-1},
    \end{align*}
    we can apply Proposition~\ref{prop:n-1-pinched-compact} to conclude that $M_t$ is either compact or flat, and the result follows as in (2).
    This is why we can slightly improve the constant from $c_n$ to $4/3n$ when $n=3$ or $4.$
    When $n=2,$ by Corollary~\ref{cor:complete-bdd-curv-pinched} and Corollary~\ref{cor:convex-then-noncollapsing}, $M_t$ is convex and non-collapsed.
    Since
    \begin{align*}
    \frac{3(n+1)}{2n(n+2)}
    = \frac 9{16}
    < 1
    = \frac 1{n-1},
    \end{align*}
    it is either flat or a uniformly strictly convex MCF.
    Thus, applying the results in \cite{BC19, ADS20} and noting that the shrinking sphere is the only uniformly strictly convex solution in the list finishes the proof.
\end{proof}

\begin{proof}
	[Proof of Corollary~\ref{cor:cyl-shrinker}]
	Let $\Sigma$ be a shrinker satisfying the assumption of Corollary~\ref{cor:cyl-shrinker}.
	The bounded entropy assumption implies \eqref{eq:volume-bounds}.
	To see  \eqref{eq:curvature-int-bounds}, we first note that the shrinker equation and the pinching condition imply that $|A|$ is of polynomial growth.
	Thus, Lemma~\ref{lem:A-poly-imply-nA-poly} and Remark~\ref{rmk:A-poly-high-codim} imply that $|\n^\perp A|$ and $|(\n^\perp)^2A|$ have polynomial growth.
	In particular, we have
	\begin{align}\label{shrinker-bdd-weighted-curvature}
	\int_\Sigma 
	\pr{|A|^2 + |\nabla^\perp A|^2 + |(\nabla^\perp)^2A|^2}
	e^{-\frac{|x|^2}4}<\infty,
	\end{align}
    which implies \eqref{eq:curvature-int-bounds} since the MCF $M_t:=\sqrt{-t}\Sigma$ generated by $\Sigma$ satisfies
	\begin{align*}
	\bd_t A_{M_t}
	= \bd_t \pr{\frac 1{\sqrt{-t}} A_\Sigma} + \ppair{\nabla A_{M_t},-\frac{x^T}{-2t}}
	= \frac 1{2(-t)^{3/2}} \pr{A_\Sigma -  \ppair{\nabla A_\Sigma, x^T}}
	\end{align*}
	where the second term comes from the tangential motion besides the scaling so that $\bd_t F$ is in the normal direction.
	Hence, Corollary~\ref{cor:complete-bdd-curv-pinched} implies that $\Sigma$ is a hypersurface.
	
	Next, Theorem~\ref{thm:improved-conv} implies that $\Sigma$ is convex.
	Strong maximum principle implies $\Sigma = \td\Sigma^k\times\bb R^{n-k}$ for some $k$ where $\td\Sigma^k$ is a strictly convex hypersurface.
    If $k=1,$ then $\td \Sigma$ is an Abresch--Langer curve.
	This forces $|A|=|H|,$ and the pinching condition is violated.
	If $k\ge 2,$ then Sacksteder's theorem~\cite{S60} implies that $\td\Sigma$ is embedded. 
	We can then apply the classification result  \cite[Theorem~0.17]{CM12} to conclude that $\Sigma$ is a generalized cylinder.
\end{proof}


\section{\textbf{Some open questions on higher codimension $c\,$-pinched MCF}}
\label{sec:Question}

In this section, we offer several questions to the community on the subject of $c$-pinched MCF in higher codimension. 

\begin{ques}
In dimension $n \geq 2$, what is the largest constant $c = c(n)$ such that every ancient, uniformly $c$-pinched solution of MCF is codimension one? 

See Appendix \ref{app:planarity} for a summary of how the constant $c_n = \min\{\frac{4}{3n}, \frac{3(n+1)}{2n(n+2)}\}$ arises in the proof of the planarity estimate. It is not known if the constant is sharp (but it seems unlikely). Obtaining the sharp constant is probably most tractable in either low dimensions or sufficiently high dimensions. In low dimensions, modified pinching conditions that involve the normal curvature and known examples (e.g. the Veronese surface and Clifford torus shrinkers) provide some clues about the sharpness of $c$-pinching \cite{BN}. In all sufficiently large dimensions, the question is somewhat analogous to an asymptotically sharp gap theorem for codimension one shrinkers. 
\end{ques}

\begin{ques}
It is possible to develop a MCF with surgery on for closed $\tilde{c}_2$-pinched solutions of MCF, where $\tilde{c}_2
:= \min \set{\frac{1}{n-2}, c_n} 
= \min \set{\frac{1}{n-2}, \frac{3(n+1)}{2n(n+2)}}$? 

These $\tilde{c}_2$-pinched flows are much like higher codimension analogues of two-convex flows. In particular, the blow-up limits of such flows are classified and must be the codimension, noncollapsed, two-convex ancients solutions. Two types of surgeries are possible: one akin to Perelman's surgery process for the Ricci flow \cite{Pe03} and the other based upon the surgery of Huisken--Sinestari \cite{HS09} (as envisioned by Hamilton \cite{Ham97}). This question has been studied by the second author in \cite{Na23}, where a canonical neighborhood theorem \`a la Perelman was developed to set the stage for a possible Perelman-type surgery. In \cite{Ngu20}, Nguyen developed many of the tools required for a Huisken--Sinestrari type surgery procedure. However, the key missing ingredient - in both approaches - is a way to preserve the planarity estimate in \cite{Na22} across surgery times. 
\end{ques}

\begin{ques} 
Is it possible to prove Theorem \ref{thm:BHS-type-est} and Theorem~\ref{thm:improved-conv} without assuming any curvature growth bounds? 
\end{ques}

Under a pinching condition, significant progress has been made on existence results for Ricci flow, even in the absence of curvature growth restrictions \cite{LT22, LT22-PIC}.
The classification theorem of~\cite{BHS} was further generalized by Yokota~\cite{Y17}, who removed the boundedness assumption on curvature.
To the best of the authors' knowledge, no analogous results are currently known in the setting of mean curvature flow.


\appendix
\section{A gradient estimate in non-convex regions}
\label{sec:app-grad-est}

The estimate in this appendix was proven in a general context in \cite{L17}. For completeness, we reprove it here for our setting.
The following lemma exploits the idea that the nonnegative difference $|\nabla A|-|\nabla |A||$ can only vanish under somewhat special conditions, like for hypersurfaces which are cylinders over curves, or when $A$ is constant. Heuristically, a condition that rules these cases should imply a definite lower bound. We exploit this idea whenever $\lambda_1 \leq -\varepsilon_0 H <0$ for $\varepsilon_0 > 0$ to obtain a useful lower bound in the proof of the convexity estimate above. 
We follow the notations in Section~\ref{sec:conv}.

\begin{lemma}
[cf. {\cite[Lemma 2.1]{L17}}]
\label{Appendix-grad-lower-bound}
On a smooth hypersurface $M,$ suppose that the smallest principal curvature $\lambda_1$ satisfies $\lambda_1 \leq -\varepsilon_0 H<0$ at a point $p\in M$ for some $\varepsilon_0 \in(0,1).$ 
Then at $p,$ we have
\begin{align*}
    {|\nabla A|^2} - |\nabla |A||^2  
    = \frac{|A\otimes \nabla A - \nabla A\otimes A|^2}{2|A|^2}
    \geq \frac{\varepsilon_0^2}{8n^2}
    \frac{|\nabla A|^2H^2}{|A|^2}. 
\end{align*}
\end{lemma}

\begin{proof}

    Working in local coordinates, we compute 
    \begin{align*}
    |A \otimes \nabla A - \nabla A \otimes A|^2 
    = 2|A|^2 |\nabla A|^2 
    - 2 \sum_{i,j,k,p,q=1}^n \pr{A_{ij} \nabla_k A_{pq}} \pr{\nabla_i A_{jk} \cdot A_{pq}}.
    \end{align*}
    By the Codazzi identity, 
    $$\sum_{i,j =1}^n A_{ij} \nabla_i A_{jk} 
    = \sum_{i,j=1}^n A_{ij} \nabla_k A_{ij} 
    = \frac{1}{2} \nabla_k |A|^2 
    = \sum_{p,q=1}^n A_{pq} \nabla_k A_{pq}.$$
    Therefore, 
    \begin{align*}
    |A \otimes \nabla A - \nabla A \otimes A|^2 
    = 2 |A|^2 |\nabla A|^2 -\frac{1}{2} |\nabla |A|^2|^2
    = 2|A|^2 \pr{|\nabla A|^2 -|\nabla |A||^2}. 
    \end{align*}
    This proves the first equality, which is generally true without the assumptions in the lemma.
    Note that here, the negative upper bound of $\lambda_1$ implies that $|A|>0.$

	Next, we choose geodesic normal coordinates that diagonalize $A$ at $p$, namely $A_{ij} = \lambda_i\delta_{ij}$. 
    Note that 
    \begin{align*}
    	|\nabla A|^2|A|^2 = \sum_{i,j,p,q,k=1}^n A_{ij}^2 \cdot |\nabla_k A_{pq}|^2.
    \end{align*}
    Since $A_{ij} = 0$ when $i\neq j$, we can simplify it into
    \begin{align*}
    	|\nabla A|^2|A|^2 = \sum_{i=1}^n \Big(A_{ii}^2\sum_{p,q,k=1}^n  |\nabla_k A_{pq}|^2\Big).
    \end{align*}

\begin{claim}\label{claim-A}
In normal coordinates that diagonalize $A$, we have
\begin{align}\label{Appendix-ineqn-11}
	 |A\otimes \nabla A - \nabla A\otimes A|^2 \geq \frac{1}{4}\sum_{i=1}^n A_{ii}^2 \pr{ - |\nabla_i A_{ii}|^2 + \sum_{p,q,k=1}^n   |\nabla_k A_{pq}|^2}.
\end{align}
\end{claim}
\begin{proof}[Proof of Claim~\ref{claim-A}]
To obtain the estimate, we express $|\nabla A|^2$ as a sum of three parts,
\[
\sum_{p,q,k =1}^n |\nabla_k A_{pq}|^2 = \sum_{\substack{1\leq k\leq n\\1\leq p\neq q\leq n}}|\nabla_k A_{pq}|^2 + \sum_{1 \leq k \neq p\leq n} |\nabla_k A_{pp}|^2 + \sum_{k =1}^n |\nabla_k A_{kk}|^2,
\]
and estimate each separately.
 
When $p\neq q$, $A_{pq} = 0$ and hence we have
$|A_{ii} \nabla_kA_{pq}| = |A_{ii} \nabla_kA_{pq} - A_{pq}\nabla_k A_{ii}|.$
Therefore,
\begin{align}\label{A-nablaA-inequality-1}
	\sum_{i=1}^n\sum_{\substack{1\leq k\leq n\\1\leq p\neq q\leq n}} A_{ii}^2 \cdot |\nabla_k A_{pq}|^2 = \sum_{i=1}^n\sum_{\substack{1\leq k\leq n\\1\leq p\neq q\leq n}} |A_{ii} \nabla_kA_{pq} - A_{pq}\nabla_k A_{ii}|^2 \leq  |A\otimes \nabla A - \nabla A\otimes A|^2.
\end{align}

Next, we assume that $p=q$ and $p \neq k$. Then $A_{kp} = 0$. Applying the Codazzi equation, we get
\begin{align*}
	|A_{ii}\nabla_k A_{pp}| = |A_{ii}\nabla_p A_{kp}| = |A_{ii} \nabla_p A_{kp} - A_{kp}\nabla_p A_{ii}|.
\end{align*}
Hence,
\begin{align}\label{A-nablaA-inequality-2}
	\sum_{i=1}^n\sum_{1\leq k\neq p\leq n} A_{ii}^2 \cdot |\nabla_k A_{pp}|^2 = \sum_{i=1}^n\sum_{1\leq k\neq p\leq n}  |A_{ii} \nabla_p A_{kp} - A_{kp}\nabla_p A_{ii}|^2 \leq  |A\otimes \nabla A - \nabla A\otimes A|^2.
\end{align}

Finally, we deal with the case when $k=p=q$. If $k\neq i$, then $A_{ik}= 0$. 
Using the Codazzi equation, we have
\begin{align*}
	|A_{ii}\nabla_k A_{kk} | 
	\leq& |A_{ii}\nabla_k A_{kk} - A_{kk}\nabla_k A_{ii}| +  |A_{kk}\nabla_k A_{ii}| \\
	=& |A_{ii}\nabla_k A_{kk} - A_{kk}\nabla_k A_{ii}| +  |A_{kk}\nabla_i A_{ik}|\\
	=& |A_{ii}\nabla_k A_{kk} - A_{kk}\nabla_k A_{ii}| + |A_{kk}\nabla_i A_{ik} - A_{ik}\nabla_i A_{kk}|.
\end{align*}
Consequently,
\begin{align}\label{A-nablaA-inequality-3}
	\sum_{i=1}^n \sum_{\substack{1\leq k\leq n\\ k\neq i}} A_{ii}^2 \cdot |\nabla_k A_{kk}|^2 \leq & \sum_{i=1}^n\sum_{\substack{1\leq k\leq n\\ k\neq i}} |A_{ii} \nabla_k A_{kk}|^2 \\
	\leq&  \sum_{i=1}^n\sum_{\substack{1\leq k\leq n\\ k\neq i}} 2|A_{ii}\nabla_k A_{kk} - A_{kk}\nabla_k A_{ii}|^2 + 2|A_{kk}\nabla_i A_{ik} - A_{ik}\nabla_i A_{kk}|^2 \nonumber\\
	\leq& 2|A\otimes \nabla A - \nabla A\otimes A|^2.\nonumber
\end{align}

Combining \eqref{A-nablaA-inequality-1}, \eqref{A-nablaA-inequality-2} and \eqref{A-nablaA-inequality-3}, we have
\begin{align*}
	&\sum_{i=1}^n A_{ii}^2\Big( - |\nabla_i A_{ii}|^2 + \sum_{p,q,k=1}^n   |\nabla_k A_{pq}|^2\Big) \\
	=& \sum_{i=1}^n A_{ii}^2 \Bigg(\sum_{\substack{1\leq k\leq n\\1\leq p\neq q\leq n}}    |\nabla_k A_{pq}|^2 +  \sum_{1\leq k\neq p\leq n}   |\nabla_k A_{pp}|^2+  \sum_{\substack{1\leq k\leq n\\ k\neq i}}    |\nabla_k A_{kk}|^2 \Bigg) \\ 
	\leq& 4|A\otimes \nabla A - \nabla A\otimes A|^2.
\end{align*}
This proves the claim.
\end{proof}

By Claim~\ref{claim-A}, to finish the proof of Lemma~\ref{Appendix-grad-lower-bound}, it remains to estimate 
$$\sum_{i=1}^n A_{ii}^2\Big( - |\nabla_i A_{ii}|^2 + \sum_{p,q,k=1}^n   |\nabla_k A_{pq}|^2\Big) $$
from below. 
Denote $\lambda_i = -A_{ii}$ to be the principal curvatures at $p$ and assume that $\lambda_1\leq \cdots\leq \lambda_n$. 
Then, by assumption of mean-convexity, we have
\begin{align}\label{lambda_n-lower-bound}
	\lambda_n \geq \frac{\sum_{i=1}^n\lambda_i}{n} = \frac{H}{n} \geq 0.
\end{align} 
Hence,
\begin{align}\label{Appendix-ineqn-12}
	&\sum_{i=1}^n A_{ii}^2 \pr{- |\nabla_i A_{ii}|^2 + \sum_{p,q,k=1}^n   |\nabla_k A_{pq}|^2} \\
    \geq&  \lambda_n^2 \pr{ - |\nabla_n A_{nn}|^2 + \sum_{p,q,k=1}^n   |\nabla_k A_{pq}|^2} 
    + \lambda_{1}^2 \pr{- |\nabla_{1} A_{11}|^2 + \sum_{p,q,k=1}^n   |\nabla_k A_{pq}|^2} \nonumber\\
    \geq& \frac{H^2}{n^2}\pr{ - |\nabla_n A_{nn}|^2 + \sum_{p,q,k=1}^n   |\nabla_k A_{pq}|^2} 
    + \varepsilon_0^2H^2 \pr{- |\nabla_{1} A_{11}|^2  + \sum_{p,q,k=1}^n   |\nabla_k A_{pq}|^2} \nonumber\\
    \geq& \frac{\varepsilon_0^2H^2}{n^2} \pr{- |\nabla_n A_{nn}|^2 - |\nabla_{1} A_{11}|^2 + 2\sum_{p,q,k=1}^n   |\nabla_k A_{pq}|^2} \nonumber\\
    \geq& \frac{\varepsilon_0^2}{n^2} H^2|\nabla A|^2 .\nonumber
\end{align}
Note that the estimate would fail if $\varepsilon_0$ were zero, because we would be unable to bound the component $|\nabla_nA_{nn}|^2$ of $|\nabla A|^2$. By \eqref{Appendix-ineqn-11} and \eqref{Appendix-ineqn-12}, we get
\begin{align*}
	|A\otimes \nabla A - \nabla A\otimes A|^2 \geq \frac{1}{4}\sum_{i=1}^n A_{ii}^2\Big( - |\nabla_i A_{ii}|^2 + \sum_{p,q,k=1}^n   |\nabla_k A_{pq}|^2\Big) \geq  \frac{\varepsilon_0^2}{4n^2}H^2 |\nabla A|^2 .
\end{align*}
The lemma then follows.
\end{proof}

\section{Dimensional constants that arise in the planarity estimate}\label{app:planarity}

In this appendix, we review how the constant $c_n$ arises in the proof of the planarity estimate. 
In particular, we summarize what estimates go into the proof and show the planarity estimate holds for $n \geq 2$. 
We use the notation introduced in Section~\ref{sec:Planarity}.

In what follows, consider a constant $c_0 > \frac{1}{n}$. Let 
\[
f := c_0 |H|^2 - |A|^2. 
\]
Assume that $|H| > 0$ so that we may write $A = h \nu_1 + \hat{A}$ and $h = \mathring{h} + \frac{1}{n} g$. Recall $R^\perp(\nu_1) = \mathring{h}_{ik} \hat{A}_{jk} - \mathring{h}_{jk} \hat{A}_{ik}$ and $\hat{R}^\perp = \hat{A}_{ik} \otimes \hat{A}_{jk} - \hat{A}_{jk} \otimes \hat{A}_{ik}$.  Recall also the definitions (from \cite{Na22}) of first-order expression
\[
Q_{ijk} := \langle \nabla^\perp_k \mathring{A}_{ij}, \nu_1 \rangle -\langle \nabla^\perp_k \hat{A}_{ij}, \nu_1 \rangle - |H|^{-1} \mathring{h}_{ij} \nabla_k |H|
\]
and the non-principal part of the normal connection $\hat{\nabla}^\perp \hat{A} := \nabla^\perp \hat{A} -\langle \nabla^\perp \hat{A} , \nu_1\rangle \nu_1$.

\subsection{The constant $\frac{4}{3n}$ and the preservation of $c$-pinching}

Direct computations (that can obtained, for instance, from the estimates in \cite{Na22}) show that
\begin{align*}
    (\partial_t - \Delta) f = 2(c_0 |\langle A, H \rangle|^2 - |\langle A, A\rangle|^2 - |R^\perp|^2) + 2(|\nabla^\perp A|^2 - c_0 |\nabla^\perp H|^2).
\end{align*}
When one has $|H| > 0$, then  
\begin{align*}
    c_0 |\langle A, H \rangle|^2 - |\langle A, A\rangle|^2 - |R^\perp|^2 & = \frac{2}{nc_0-1} |\hat{A}|^2f + \frac{nc_0}{nc_0 - 1} |\mathring{h}|^2 f + \frac{1}{nc_0-1} f^2 \\
    & \quad + 2\left( 2|\mathring{h}|^2 |\hat{A}|^2 - |\mathring{h}_{ij} \hat{A}_{ij}|^2 - |R^\perp (\nu_1)|^2 \right) \\
    & \quad + \left( \frac{3}{2}|\hat{A}|^4 - |\langle \hat{A}, \hat{A} \rangle|^2 - |\hat{R}^\perp|^2 \right)\\
    & \quad + \left(\frac{1}{nc_0 -1} - \frac{3}{2}\right)|\hat{A}|^4 +  \left(\frac{nc_0}{nc_0 -1}-4\right)|\mathring{h}|^2 |\hat{A}|^2 , \\
|\nabla^\perp A|^2 - c_0 |\nabla^\perp H|^2& =   c_0\left(\frac{n+2}{3} |\nabla^\perp A|^2 - |\nabla^\perp H|^2 \right) + \left(1 - \frac{n+2}{3}c_0  \right)|\nabla^\perp A|^2.
\end{align*}
Estimates of Huisken \cite{H84} and Andrews-Baker \cite{AB} show that for any $n \geq 2$:
\begin{align*}
    \frac{n+2}{3} |\nabla^\perp A|^2 - |\nabla^\perp H|^2 &\geq 0 , \\
    2|\mathring{h}|^2 |\hat{A}|^2 - |\mathring{h}_{ij} \hat{A}_{ij}|^2 - |R^\perp (\nu_1)|^2& \geq 0, \\
    \frac{3}{2}|\hat{A}|^4 - |\langle \hat{A}, \hat{A} \rangle|^2 - |\hat{R}^\perp|^2 & \geq 0. 
\end{align*}
It follows that $(\partial_t - \Delta ) f \geq 0$ for all $n \geq 2$ if 
\begin{align*}
    1 -\frac{n+2}{3}c_0 &\geq 0  \iff c_0 \leq \frac{3}{n+2}, \\
        \frac{nc_0}{nc_0-1} - 4 & \geq 0 \iff c_0 \leq \frac{4}{3n}, \\
   \frac{1}{nc_0 - 1} - \frac{3}{2} & \geq 0 \iff c_0 \leq \frac{5}{3n}.
\end{align*}
Imposing the strictest bound is how $c_0 \leq \frac{4}{3n}$ for $n \geq 2$ arises. 
\subsection{The constant $\frac{3(n+1)}{2n(n+2)}$ and the planarity estimate}

Now let $u := \frac{|\hat{A}|^2}{f}$. Then
\begin{align*}
(\partial_t -2\nabla_{\nabla f}- \Delta )u & =- \frac{2}{f} \left( u(c_0 |\langle A, H \rangle|^2 - |\langle A, A\rangle|^2 - |R^\perp|^2) - |\langle \hat{A}, \hat{A} \rangle|^2 - |\hat{R}^\perp|^2 - |R^\perp (\nu_1)|^2 \right)  \\
& \quad - \frac{2}{f}\left(  |\nabla^\perp \hat{A}|^2 +  u(|\nabla^\perp A|^2 - c_0 |\nabla^\perp H|^2)-2 Q_{ijk}\langle \hat{A}_{ij}, \nabla^\perp_k\nu_1 \rangle \right).
\end{align*}

To show $(\partial_t - 2\nabla_{\nabla f} - \Delta) u\leq0$, we estimate the reaction and gradient terms separately. For the reaction terms, we use the expansion of the previous section and collect the reaction terms coming from the evolution of $|\hat{A}|^2$ to obtain the idenity:
\begin{align*}
    &\quad u(c_0 |\langle A, H \rangle|^2 - |\langle A, A\rangle|^2 - |R^\perp|^2) - |\langle \hat{A}, \hat{A} \rangle|^2 -|\hat{R}^\perp|^2 - |R^\perp (\nu_1)|^2\\
    &= \frac{1}{nc_0-1} f|\hat{A}|^2  +|\mathring{h}_{ij}\hat{A}_{ij}|^2 + \frac{3}{2} |\hat{A}|^4 + 2|\mathring{h}|^2 |\hat{A}|^2 \\
    & \quad  +(2u + 1)\left( 2|\mathring{h}|^2 |\hat{A}|^2 - |\mathring{h}_{ij} \hat{A}_{ij}|^2 - |R^\perp (\nu_1)|^2 \right) \\
    & \quad + (u+1)\left( \frac{3}{2}|\hat{A}|^4 - |\langle \hat{A}, \hat{A} \rangle|^2 - |\hat{R}^\perp|^2 \right)\\
    & \quad +(u+2)\left(\frac{1}{nc_0 -1} - \frac{3}{2}\right)|\hat{A}|^4 + (u +1) \left(\frac{nc_0}{nc_0 -1}-4\right)|\mathring{h}|^2 |\hat{A}|^2.
\end{align*}
We arranged terms to impose the same constraint on \( c_0 \) as was required for positivity of \((\partial_t - \Delta) f\). However, the identity shows that controlling the negative terms from the evolution of \( |\hat{A}|^2 \) (those without a factor of $u$) actually requires less, thanks to the leftover positive terms on the first line.

For the gradient terms, following arguments in \cite{Na22}, for any $n \geq 2$ have the identity
\begin{align*}
    |\nabla^\perp \hat{A}|^2 
    & = \left(|\hat{\nabla}^\perp \hat{A}|^2 + |\mathring{h}|^2 |\nabla^\perp\nu_1|^2 - \frac{1}{2} |\hat{\nabla}^\perp \hat{A} + \mathring{h} \nabla^\perp \nu_1|^2\right)\\
    & \quad + \left( \frac{1}{2} |\hat{\nabla}^\perp \hat{A} + \mathring{h} \nabla^\perp \nu_1|^2 - \frac{n-1}{n(n+2)} |H|^2 |\nabla^\perp \nu_1|^2\right)\\
   & \quad + |\langle \nabla^\perp \hat{A}, \nu_1\rangle|^2 +\left( \frac{n-1}{(n+2)(nc_0-1)} \right) ( f +|\hat{A}|^2)|\nabla^\perp \nu_1|^2\\
   & \quad + \left( \frac{n-1}{(n+2)(nc_0-1)}-1 \right)  |\mathring{h}|^2|\nabla^\perp \nu_1|^2,
\end{align*}
which, using the improved gradient estimate (see \cite[(4.22)]{Na22}) and Cauchy's inequality, gives the estimate 
\begin{align}\label{eq:B1}
     |\nabla^\perp \hat{A}|^2 & \geq |\langle \nabla^\perp \hat{A}, \nu_1\rangle|^2 +\left( \frac{n-1}{(n+2)(nc_0-1)} \right) ( f +|\hat{A}|^2)|\nabla^\perp \nu_1|^2\\
\nonumber   & \quad+ \left( \frac{n-1}{(n+2)(nc_0-1)}-1 \right)  |\mathring{h}|^2|\nabla^\perp \nu_1|^2.
\end{align}
Similarly, we have the identity
\begin{align*}
     u(|\nabla^\perp A|^2 -  c_0 |\nabla^\perp H|^2) & = u\left(c_0 - \frac{1}{n} \right)\left(\frac{n(n+2)}{2(n-1)}|\langle \nabla^\perp \mathring{A},\nu_1 \rangle|^2  - |\nabla|H||^2\right)\\
     & \quad + u \left(|\hat{\nabla}^\perp \hat{A} + h \nabla^\perp \nu_1|^2  - \frac{3}{n+2} |H|^2 |\nabla^\perp \nu_1|^2 \right)\\
     & \quad + \frac{n}{nc_0-1} \Big(\frac{3}{n+2}-c_0\Big) (u|\hat{A}|^2 + u|\mathring{h}|^2+ |\hat{A}|^2) |\nabla^\perp \nu_1|^2 \\
     & \quad + \left(1 - \frac{(n+2)(nc_0-1)}{2(n-1)}\right)u|\langle \nabla^\perp \mathring{A},\nu_1 \rangle|^2, 
\end{align*}
which (see \cite[(4.20,4.21)]{Na22}) gives the estimate 
\begin{align}
    u(|\nabla^\perp A|^2 -  c_0 |\nabla^\perp H|^2) & \geq \frac{n}{nc_0-1} \Big(\frac{3}{n+2}-c_0\Big) (u|\hat{A}|^2 + u|\mathring{h}|^2+ |\hat{A}|^2) |\nabla^\perp \nu_1|^2 \\ 
  \nonumber   & \quad+ \left(1 - \frac{(n+2)(nc_0-1)}{2(n-1)}\right)u|\langle \nabla^\perp \mathring{A},\nu_1 \rangle|^2.
\end{align}
Finally, we use Young's inequality with constants $a_1, a_2, a_3 > 0$ to estimate
\begin{align}
    2 Q_{ijk}\langle \hat{A}_{ij}, \nabla^\perp_k\nu_1 \rangle & \leq a_2 u|\langle \nabla^\perp \mathring{A}, \nu_1 \rangle|^2 + \frac{1}{a_2} f |\nabla^\perp \nu_1|^2 + a_1 |\langle \nabla^\perp \hat{A}, \nu_1 \rangle|^2 +\frac{1}{a_1} |\hat{A}|^2 |\nabla^\perp \nu_1|^2\\
\nonumber   & \quad + a_3 \frac{|\hat{A}|^2}{|H|^2} |\nabla |H||^2 + \frac{1}{a_3} |\mathring{h}|^2 |\nabla^\perp \nu_1|^2,
\end{align}
and we bound the fifth term on the right hand side  by
\begin{equation}\label{eq:B4}
     a_3 \frac{|\hat{A}|^2}{|H|^2} |\nabla |H||^2  \leq a_3 \frac{(n+2)(nc_0-1)}{2(n-1)}u|\langle \nabla^\perp \mathring{A},\nu_1 \rangle|^2.
\end{equation}
Putting \eqref{eq:B1} through \eqref{eq:B4} together gives 
\begin{align*}
   & \quad |\nabla^\perp \hat{A}|^2 +  u(|\nabla^\perp A|^2 - c_0 |\nabla^\perp H|^2)-2 Q_{ijk}\langle \hat{A}_{ij}, \nabla^\perp_k\nu_1 \rangle\\ 
   & \geq (1-a_1)|\langle \nabla^\perp \hat{A}, \nu_1\rangle|^2 +\Big(\frac{4n-1}{(n+2)(nc_0-1)} -\frac{nc_0}{nc_0-1}-\frac{1}{a_1}\Big)  |\hat{A}|^2 |\nabla^\perp \nu_1|^2\\   
   & \quad +\left( \frac{n-1}{(n+2)(nc_0-1)} - \frac{1}{a_2} \right)f|\nabla^\perp \nu_1|^2 \\
 &  \quad + \left( \frac{n-1}{(n+2)(nc_0-1)}-1 - \frac{1}{a_3} \right)  |\mathring{h}|^2|\nabla^\perp \nu_1|^2 \\ 
   & \quad + \left(1 - (1+a_3)\frac{(n+2)(nc_0-1)}{2(n-1)} - a_2\right)u|\langle \nabla^\perp \mathring{A},\nu_1 \rangle|^2\\
  & \quad + \frac{n}{nc_0-1} \Big(\frac{3}{n+2}-c_0\Big) (u|\hat{A}|^2 + u|\mathring{h}|^2) |\nabla^\perp \nu_1|^2. 
\end{align*}
To make the inequality as sharp as possible, we see that one should take $a_1 = 1$. Then nonnegativity of the coefficient of $|\hat{A}|^2 |\nabla^\perp \nu_1|^2$ requires $c_0 \leq \frac{5n +1}{2n(n+2)}$ and nonnegativity of the term on the last line requires $c_0 \leq \frac{3}{n+2}$. The remaining coefficients involving $a_2, a_3$ require 
\begin{align*}
   1 - (1+a_3)\frac{(n+2)(nc_0-1)}{2(n-1)}-a_2 &\geq 0, \iff \frac{2(n-1)}{n(n+2)} \frac{1 -a_2}{1+a_3} +\frac{1}{n}\geq c_0\iff 2\frac{1-a_2}{1+a_3} \geq \frac{C_0}{C_n}, \\
    \frac{n-1}{(n+2)(nc_0-1)} - \frac{1}{a_2} & \geq 0 \iff  \frac{n-1}{n(n+2)}a_2 +  \frac{1}{n} \geq c_0\iff a_2 \geq \frac{C_0}{C_n},  \\
     \frac{n-1}{(n+2)(nc_0-1)}-1-\frac{1}{a_3} & \geq 0 \iff \frac{n-1}{n(n+2)} \frac{a_3}{1+a_3} + \frac{1}{n}\geq c_0  \iff \frac{a_3}{1+a_3} \geq \frac{C_0}{C_n},
\end{align*}
where $C_0:= c_0-1/n$ and $C_n:=\frac{n-1}{n(n+2)}.$ 
The maximum of $F(a_2, a_3) := \min\{a_2, 2\frac{1-a_2}{1+a_3}, \frac{a_3}{1+a_3}\}$ for $a_2, a_3 > 0$ occurs when $a_2 = \frac{1}{2}$ and $a_3 = 1$ and is $F(\frac{1}{2}, 1) = \frac{1}{2}$.
Therefore, we conclude that, in addition to the bounds above, we need that $c_0 \leq \frac{1}{n} + \frac{n-1}{2n(n+2)} = \frac{3(n+1)}{2n(n+2)}.$ Note this is always smaller than $\frac{5n+1}{2n(n+2)}$ and $\frac{3}{n+2}$ for all $n \geq 2$. 

In conclusion, the reaction terms and gradients have the correct sign to obtain the planarity estimate for all $n \geq 2$ under the assumption that 
\[
c_0 \leq c_n:=\min \left\{\frac{4}{3n}, \frac{3(n+1)}{2n(n+2)}\right\} = \begin{cases}
    \frac{3(n+1)}{2n(n+2)},  & n \leq 7 \\ \frac{4}{3n}, & n \geq 8
\end{cases}.
\]


\end{document}